\documentclass[leqno,12pt]{amsart}
\usepackage{amssymb}
\usepackage{amsmath}
\usepackage{enumerate}
\usepackage{amsfonts}
\usepackage{hyperref}
\usepackage{mathrsfs}
\usepackage{tikz}

\hypersetup{
    colorlinks=true,
    linkcolor= blue,
    citecolor =cyan,
    urlcolor = teal,
}

\newtheorem{theorem}{Theorem}[section]
\newtheorem{corollary}[theorem]{Corollary}
\newtheorem{lemma}[theorem]{Lemma}
\newtheorem{proposition}[theorem]{Proposition}

\newtheorem{definition}[theorem]{Definition}
\newtheorem{fact}[theorem]{Fact}

\numberwithin{equation}{section}

\newcommand{\supp}{\text{\rm supp}\,}

\begin{document}
\title[Dunkl-Schr\"odinger operators and Fefferman--Phong inequality] {Schr\"odinger operators\\  with reverse H\"older class potentials   \\  in the  Dunkl setting and their Hardy spaces}

\author[Agnieszka Hejna]{Agnieszka Hejna}


\begin{abstract}
For a normalized root system $R$ in $\mathbb R^N$ and a multiplicity function $k\geq 0$ let $\mathbf N=N+\sum_{\alpha \in R} k(\alpha)$.
Let $L=-\Delta +V$, $V\geq 0$,  be the Dunkl--Schr\"odinger operator on $\mathbb R^N$. Assume that there exists $q >\max(1,\frac{\mathbf{N}}{2})$ such that  $V$ belongs to the reverse H\"older class $\text{RH}^q(dw)$. We prove the Fefferman--Phong inequality for $L$. As an application, we conclude  that the Hardy space $H^1_{L}$, which is originally defined by means of the maximal function associated with the semigroup $e^{tL}$, admits an atomic decomposition with local atoms in the sense of Goldberg, where their localization are adapted to  $V$.
\end{abstract}
\address{A. Hejna, Uniwersytet Wroc\l awski,
Instytut Matematyczny,
Pl. Grunwaldzki 2/4,
50-384 Wroc\l aw,
Poland}
\email{hejna@math.uni.wroc.pl}

\thanks{
Research supported by the National Science Centre, Poland (Narodowe Centrum Nauki), Grant 2017/25/B/ST1/00599}

\enlargethispage{0.37cm}
\maketitle
\tableofcontents
\thispagestyle{empty}

\section{Introduction}
On  $\mathbb R^N$, $N\geq 3$, let us consider the Schr\"odinger differential operator
\begin{equation}\label{eq:operator_classical}
\mathscr{L}=   -\Delta_{\text{eucl}}+V(x)=-\sum_{j=1}^{N}\partial_{j}^2+V(x)
\end{equation}
where $V \in L^2_{\text{loc}}(\mathbb{R}^N,\,dx)$ is a non-negative potential which $V$ belongs to the reverse H\"older class $B_{q}$ with  $q>\frac{N}{2}$, i.e. the  inequality 
\begin{equation}\label{eq:reverse_classic}
    \left(\frac{1}{|B|}\int_{B}V(x)^{q}\,dx\right)^{1/q} \leq C \frac{1}{|B|}\int_{B}V(x)\,dx
\end{equation}
holds for every ball $B$ in $\mathbb{R}^N$. Define the auxiliary function $\mathbf{m}$ as follows:
\begin{equation}\label{eq:m_classic}
    \frac{1}{\mathbf{m}(x)}=\sup\left\{r>0\,:\,\frac{1}{r^{n-2}}\int_{B(x,r)}V(x)\,dx  \leq 1\right\}.
\end{equation}
The integral defining the function $\mathbf{m}$  was introduced  by Ch. Fefferman (see \cite[p. 146, the assumption of the main lemma]{Fefferman}). The function  is then  used in the well-known Fefferman--Phong inequality (\cite[p. 146]{Fefferman}, see also Shen ~\cite{Shen2},~\cite[Lemma 1.9]{Shen}) which we state  below.
\begin{theorem}[Fefferman--Phong inequality]
There is a constant $C>0$ such that for all $f \in C^1_{c}(\mathbb{R}^N)$ we have
\begin{equation}\label{eq:Fefferman-Phong_classic}
    \int_{\mathbb{R}^N}\mathbf{m}(x)^2|f(x)|^2\,dx \leq C \left(\sum_{j=1}^{N}\int_{\mathbb{R}^N}|\partial_{j}f(x)|^2\,dx+\int_{\mathbb{R}^N}V(x)|f(x)|^2\,dx\right).
\end{equation}
\end{theorem}
The proof of~\eqref{eq:Fefferman-Phong_classic} is based on the usage of the fact that $V \in A_p$ for some $p>1$ and the Poincar\'e  inequality
\begin{equation}\label{eq:Poincare}
    \frac{1}{|B(x,r)|}\int_{B(x,R)}|f(y)-f_{B(x,r)}|^2\,dy \leq \frac{Cr^2}{|B(x,r)|}\int_{B(x,r)}|\nabla f(y)|^2\,dx.
\end{equation}

The Fefferman--Phong inequality and the function $\mathbf m$ itself are    very useful tools  which are used in  analysis regarding the operator $\mathscr{L}$, e.g.,  in investigating  behavior of its eigenvalues~\cite{Fefferman}, estimating  of the  fundamental solution of the equation $\mathscr Lu=0$ (\cite[Theorem 2.7]{Shen}) and  studying $L^p$-bounds of the operators 
$\nabla \mathscr{L}^{i\gamma}, \, \nabla \mathscr{L}^{-1/2}, \nabla \mathscr{L}^{-1}\nabla, \, \nabla^2 \mathscr{L}^{-1}$ (see Theorems 0.3, 0.4, 0.5, 0.8 in~\cite{Shen}). It was proved  in \cite{ DZ_Revista} (see also~\cite[Theorem 2.11, Proposition 2.16]{DZ_RZM}) that the integral kernel $k_t(x,y)$ of the Schr\"odinger semigroup $e^{-t\mathscr{L}}$ behaves like the classical heat semigroup for $0<t<\mathbf m(x)^{-2}$, while for $t>\mathbf m(x)^{-2}$ has essentially faster decay. These observations allowed Dziubański and Zienkiewicz ~\cite{DZ_Revista} to study the Hardy spaces associated with $\mathscr{L}$ and prove a local character of atoms (see also \cite{DZ_Coll, DZ_Studia}).

The aim of this article is to prove the  Fefferman--Phong inequality for  Dunkl--Schr\"odinger operators and study its applications for describing behavior of the corresponding Dunkl-Schr\"odinger semigroups and their  Hardy spaces  $H^1$.

The Dunkl theory is a generalization of the Euclidean Fourier analysis. It started with the seminal article \cite{Dunkl} and developed extensively afterwards (see e.g. \cite{RoeslerDeJeu}, \cite{Dunkl0}, \cite{Dunkl3}, \cite{Dunkl2}, \cite{GR}, \cite{Roesler2}, \cite{Roesle99}, \cite{Roesler2003}, \cite{ThangaveluXu}, \cite{Trimeche2002}). We refer the reader to lecture notes~\cite{Roesler3} and~\cite{Roesler-Voit}  for more information and references. We fix a  normalized  root system $R$ in $\mathbb R^N$  and a multiplicity function $k\geq 0$ (see Section \ref{sec:preliminaries}). For $\xi \in \mathbb{R}^N$, $N \geq 1$, the Dunkl operators $T_\xi$  are the following $k$-deformations of the directional derivatives $\partial_\xi$ by a  difference operator:
\begin{equation}\label{eq:Dunkl_operator}
     T_\xi f(\mathbf x)= \partial_\xi f(\mathbf x) + \sum_{\alpha\in R} \frac{k(\alpha)}{2}\langle\alpha ,\xi\rangle\frac{f(\mathbf x)-f(\sigma_\alpha (\mathbf x))}{\langle \alpha,\mathbf x\rangle},
\end{equation}
where $\sigma_\alpha$ is the reflection on $\mathbb R^N$ with respect to the hyperspace orthogonal to $\alpha$. The Dunkl operators are generalizations of the partial derivatives (in fact, they are ordinary partial derivatives for $k \equiv 0$), however they are non-local operators. Therefore, in order to obtain counterparts of  classical Euclidean  harmonic analysis  results in the Dunkl setting,  we have to deal with both: local and non-local parts of the operators under consideration. For instance, the question what would be a good   counterpart of Poincare's inequality~\eqref{eq:Poincare} is true in the rational Dunkl setting seems to be an interesting   problem. Recently various different versions of~\eqref{eq:Poincare} were proved (see~\cite{Maslouhi},~\cite{Velicu1},~\cite{Velicu2}).  
The analysis is more complicated if we compose such operators. Furthermore, there are other technical problems and open questions in Dunkl theory. One of them is the  lack of knowledge about boundendess of the so called Dunkl translations $\tau_{\mathbf x}$ on $L^p(dw)$-spaces for $p\ne 2$. It makes analysis of convolution operators more complicated and delicate.

In the present paper we consider the Dunkl--Schr\"odinger operator
\begin{align*}
   L= -\Delta + V \text{ on }\mathbb{R}^N, \, N \geq 1,
\end{align*}
where $V \in L^2_{\text{loc}}(dw)$ is non-negative potential and $\Delta=\sum_{j=1}^N T_{e_j}^2$ is the Dunkl Laplacian. Such operators were recently studied by Amri and Hammi in~\cite{AH} and ~\cite{AH_2}. An example of such operator is the so called Dunkl harmonic oscillator $-\Delta +\| x\|^2$, whose properties  are  better understood (see~\cite{Amri15},~\cite{Hejna},~\cite{Nowak}, ~\cite{Roesler2}, and~\cite{Roesler-Voit}). Let $\mathbf{N}$ be the homogeneous dimension (see~\eqref{eq:homo}). We shall assume that $V$ satisfies an analogue of~\eqref{eq:reverse_classic} with $q>\max(1,\frac{\mathbf{N}}{2})$ (see Subsection~\ref{sec:Schorodinger} for details). In the current paper we prove that a counterpart of the Fefferman--Phong inequality ~\eqref{eq:Fefferman-Phong_classic} is true in the Dunkl setting, which is one of  our main results  (see Theorem~\ref{theo:Fefferman-Phong}). The main difficulty which one faces trying to prove Theorem~\ref{theo:Fefferman-Phong} is the lack of knowledge about the Poincare's inequality, which is the main ingredient of the proof in the classical case. Our idea of the proof is to mix the methods which are known from the theory of non-local operator (see~\cite[proof of Theorem 9.4]{DZ_Studia}), a version of pseudo--Poincare's inequality (which is very close to that in~\cite[Section 5]{Velicu1}), together with a careful analysis of properties of the counterpart of the function $\mathbf{m}$ compared to the structure of the Dunkl operator. The analysis of properties of the counterpart of the function $\mathbf{m}$ (see~\eqref{eq:m}) and the proof of Theorem~\ref{theo:Fefferman-Phong} are the goals of Part~\ref{part:Fefferman-Phong} of the paper.

Part~\ref{part:Hardy} is devoted to the application of the Fefferman--Phong inequality to prove the characterization of the Hardy space $H^{1}_{L}$ associated with the Dunkl--Schr\"odinger operator by the maximal function associated with the semigroup generated by $-\Delta+V$ and by a special atomic decomposition - see Section~\ref{sec:Hardy_pre} for details. This application is inspired by~\cite{DZ_Revista} (see also~\cite{DZ_Studia_2} and~\cite{DZ_Coll}). The atoms for $H^1_L$ have the structure of local atoms in the sense of Goldberg \cite{Goldberg}   with localization adapted to the behavior of the function $m$. So,  in order to obtain our result, we  need  characterizations of a family  local Hardy spaces in the Dunkl setting proved in  ~\cite[Section 5]{Hejna}.

\section{Preliminaries}\label{sec:preliminaries}

\subsection{The basic definitions of the Dunkl theory}
In this section we present basic facts concerning the theory of the Dunkl operators.  For details we refer the reader to~\cite{Dunkl},~\cite{Roesler3}, and~\cite{Roesler-Voit}. 

We consider the Euclidean space $\mathbb R^N$ with the scalar product $\langle\mathbf x,\mathbf y\rangle=\sum_{j=1}^N x_jy_j
$, where $\mathbf x=(x_1,...,x_N)$, $\mathbf y=(y_1,...,y_N)$, and the norm $\| \mathbf x\|^2=\langle \mathbf x,\mathbf x\rangle$. For a nonzero vector $\alpha\in\mathbb R^N$,  the reflection $\sigma_\alpha$ with respect to the hyperplane $\alpha^\perp$ orthogonal to $\alpha$ is given by
\begin{align*}
\sigma_\alpha (\mathbf x)=\mathbf x-2\frac{\langle \mathbf x,\alpha\rangle}{\| \alpha\| ^2}\alpha.
\end{align*}
In this paper we fix a normalized root system in $\mathbb R^N$, that is, a finite set  $R\subset \mathbb R^N\setminus\{0\}$ such that $R \cap \alpha \mathbb{R} = \{\pm \alpha\}$,  $\sigma_\alpha (R)=R$, and $\|\alpha\|=\sqrt{2}$ for all $\alpha\in R$. The finite group $G$ generated by the reflections $\sigma_\alpha \in R$ is called the {\it Weyl group} ({\it reflection group}) of the root system. A~{\textit{multiplicity function}} is a $G$-invariant function $k:R\to\mathbb C$ which will be fixed and $\geq 0$  throughout this paper. 
 Let
\begin{equation}\label{eq:measure_formula}
dw(\mathbf x)=\prod_{\alpha\in R}|\langle \mathbf x,\alpha\rangle|^{k(\alpha)}\, d\mathbf x
\end{equation} 
be  the associated measure in $\mathbb R^N$, where, here and subsequently, $d\mathbf x$ stands for the Lebesgue measure in $\mathbb R^N$.
We denote by 
\begin{equation}\label{eq:homo}
\mathbf{N}=N+\sum_{\alpha \in R} k(\alpha)
\end{equation}
the homogeneous dimension of the system. Clearly, 
\begin{align*} w(B(t\mathbf x, tr))=t^{\mathbf N}w(B(\mathbf x,r)) \ \ \text{\rm for all } \mathbf x\in\mathbb R^N, \ t,r>0,   
\end{align*}
 where $B(\mathbf x, r)=\{\mathbf y\in\mathbb R^N: \|\mathbf y-\mathbf x\|<r\}$. Moreover, 
\begin{align*}
\int_{\mathbb R^N} f(\mathbf x)\, dw(\mathbf x)=\int_{\mathbb R^N} t^{-\mathbf N} f(\mathbf x\slash t)\, dw(\mathbf x)\ \ \text{for} \ f\in L^1(dw)  \   \text{\rm and} \  t>0.
\end{align*}
Observe that there is a constant $C>0$ such that 
\begin{equation}\label{eq:balls_asymp} 
C^{-1}w(B(\mathbf x,r))\leq  r^{N}\prod_{\alpha \in R} (|\langle \mathbf x,\alpha\rangle |+r)^{k(\alpha)}\leq C w(B(\mathbf x,r)),
\end{equation}
so $dw(\mathbf x)$ is doubling, that is, there is a constant $C>0$ such that
\begin{equation}\label{eq:doubling} w(B(\mathbf x,2r))\leq C w(B(\mathbf x,r)) \ \ \text{ for all } \mathbf x\in\mathbb R^N, \ r>0.
\end{equation}
Moreover, there exists a constant $C\ge1$ such that,
for every $\mathbf{x}\in\mathbb{R}^N$ and for every $r_2\ge r_1>0$,
\begin{equation}\label{eq:growth}
C^{-1}\Big(\frac{r_2}{r_1}\Big)^{N}\leq\frac{{w}(B(\mathbf{x},r_2))}{{w}(B(\mathbf{x},r_1))}\leq C \Big(\frac{r_2}{r_1}\Big)^{\mathbf{N}}.
\end{equation}

For a measurable subset $A$ of $\mathbb{R}^N$ we define 
\begin{equation}\label{eq:orbit_of_A}
    \mathcal{O}(A)=\{\sigma_{\alpha}(\mathbf{x})\,:\, \mathbf{x} \in A, \, \alpha \in R\}.
\end{equation}
Clearly, by~\eqref{eq:balls_asymp}, for all $\mathbf{x} \in \mathbb{R}^N$ and $r>0$ we get
\begin{equation}\label{eq:ball_orbit_compare}
    w(\mathcal{O}(B(\mathbf{x},r))) \leq |G|w(B(\mathbf{x},r)).
\end{equation}

For $\xi \in \mathbb{R}^N$, the {\it Dunkl operators} $T_\xi$  are the following $k$-deformations of the directional derivatives $\partial_\xi$ by a  difference operator:
\begin{equation}\label{eq:T_def}
     T_\xi f(\mathbf x)= \partial_\xi f(\mathbf x) + \sum_{\alpha\in R} \frac{k(\alpha)}{2}\langle\alpha ,\xi\rangle\frac{f(\mathbf x)-f(\sigma_\alpha(\mathbf{x}))}{\langle \alpha,\mathbf x\rangle}.
\end{equation}
The Dunkl operators $T_{\xi}$, which were introduced in~\cite{Dunkl}, commute and are skew-symmetric with respect to the $G$-invariant measure $dw$.

For fixed $\mathbf y\in\mathbb R^N$ the {\it Dunkl kernel} $E(\mathbf x,\mathbf y)$ is the unique analytic solution to the system
\begin{equation}\label{eq:Dunkl_kernel_definition}
    T_\xi f=\langle \xi,\mathbf y\rangle f, \ \ f(0)=1.
\end{equation}
The function $E(\mathbf x ,\mathbf y)$, which generalizes the exponential  function $e^{\langle \mathbf x,\mathbf y\rangle}$, has the unique extension to a holomorphic function on $\mathbb C^N\times \mathbb C^N$. Moreover, it satisfies $E(\mathbf{x},\mathbf{y})=E(\mathbf{y},\mathbf{x})$ for all $\mathbf{x},\mathbf{y} \in \mathbb{C}^N$.

Let $\{e_j\}_{1 \leq j \leq N}$ denote the canonical orthonormal basis in $\mathbb R^N$ and let $T_j=T_{e_j}$. 
In our further consideration we shall need the following lemma.
\begin{lemma}
For all $\mathbf{x} \in \mathbb{R}^N$, $\mathbf{z} \in \mathbb{C}^N$ and $\nu \in \mathbb{N}_0^{N}$ we have
$$|\partial^{\nu}_{\mathbf{z}}E(\mathbf{x},\mathbf{z})| \leq \|\mathbf{x}\|^{|\nu|}\exp(\|\mathbf{x}\|\|{\rm Re \;}\mathbf{z}\|).$$
In particular, 
\begin{align*} | E(i\xi, \mathbf x)|\leq 1 \quad \text{ for all } \xi,\mathbf x\in \mathbb R^N.
\end{align*}
\end{lemma}
\begin{proof}
See~\cite[Corollary 5.3]{Roesle99}.
\end{proof}

\begin{corollary}\label{coro:Roesler}
There is a constant $C>0$ such that for all $\mathbf{x},\xi \in \mathbb{R}^N$ we have
\begin{equation}
|E(i\xi,\mathbf{x})-1| \leq C\|\mathbf{x}\|\|\xi\|.
\end{equation}
\end{corollary}

The \textit{Dunkl transform}
  \begin{align*}\mathcal F f(\xi)=c_k^{-1}\int_{\mathbb R^N} E(-i\xi, \mathbf x)f(\mathbf x)\, dw(\mathbf x),
  \end{align*}
  where
  $$c_k=\int_{\mathbb{R}^N}e^{-\frac{\|\mathbf{x}\|^2}{2}}\,dw(\mathbf{x})>0,$$
   originally defined for $f\in L^1(dw)$, is an isometry on $L^2(dw)$, i.e.,
   \begin{equation}\label{eq:Plancherel}
       \|f\|_{L^2(dw)}=\|\mathcal{F}f\|_{L^2(dw)} \text{ for all }f \in L^2(dw),
   \end{equation}
and preserves the Schwartz class of functions $\mathcal S(\mathbb R^N)$ (see \cite{deJeu}). Its inverse $\mathcal F^{-1}$ has the form
  \begin{align*} \mathcal F^{-1} g(x)=c_k^{-1}\int_{\mathbb R^N} E(i\xi, \mathbf x)g(\xi)\, dw(\xi).
  \end{align*}
The {\it Dunkl translation\/} $\tau_{\mathbf{x}}f$ of a function $f\in\mathcal{S}(\mathbb{R}^N)$ by $\mathbf{x}\in\mathbb{R}^N$ is defined by
\begin{align*}
\tau_{\mathbf{x}} f(\mathbf{y})=c_k^{-1} \int_{\mathbb{R}^N}{E}(i\xi,\mathbf{x})\,{E}(i\xi,\mathbf{y})\,\mathcal{F}f(\xi)\,{dw}(\xi).
\end{align*}
  It is a contraction on $L^2(dw)$, however it is an open  problem  if the Dunkl translations are bounded operators on $L^p(dw)$ for $p\ne 2$.
  
  {The \textit{Dunkl convolution\/} $f*g$ of two reasonable functions (for instance Schwartz functions) is defined by
$$
(f*g)(\mathbf{x})=c_k\,\mathcal{F}^{-1}[(\mathcal{F}f)(\mathcal{F}g)](\mathbf{x})=\int_{\mathbb{R}^N}(\mathcal{F}f)(\xi)\,(\mathcal{F}g)(\xi)\,E(\mathbf{x},i\xi)\,dw(\xi) \text{ for }\mathbf{x}\in\mathbb{R}^N,
$$
or, equivalently, by}
\begin{align*}
  {(f{*}g)(\mathbf{x})=\int_{\mathbb{R}^N}f(\mathbf{y})\,\tau_{\mathbf{x}}g(-\mathbf{y})\,{dw}(\mathbf{y})=\int_{\mathbb R^N} f(\mathbf y)g(\mathbf x,\mathbf y) \,dw(\mathbf{y}) \text{ for all } \mathbf{x}\in\mathbb{R}^N},  
\end{align*}
where, here and subsequently, $g(\mathbf x,\mathbf y)=\tau_{\mathbf x}g(-\mathbf y)$. 

\subsection{Dunkl Laplacian and Dunkl heat semigroup} The {\it Dunkl Laplacian} associated with $R$ and $k$  is the differential-difference operator $\Delta=\sum_{j=1}^N T_{j}^2$, which  acts on $C^2(\mathbb{R}^N)$-functions by

\begin{align*}
    \Delta f(\mathbf x)=\Delta_{\rm eucl} f(\mathbf x)+\sum_{\alpha\in R} k(\alpha) \delta_\alpha f(\mathbf x),
\end{align*}
\begin{align*}
    \delta_\alpha f(\mathbf x)=\frac{\partial_\alpha f(\mathbf x)}{\langle \alpha , \mathbf x\rangle} - \frac{\|\alpha\|^2}{2} \frac{f(\mathbf x)-f(\sigma_\alpha \mathbf x)}{\langle \alpha, \mathbf x\rangle^2}.
\end{align*}
Obviously, $\mathcal F(\Delta f)(\xi)=-\| \xi\|^2\mathcal Ff(\xi)$. The operator $\Delta$ is essentially self-adjoint on $L^2(dw)$ (see for instance \cite[Theorem\;3.1]{AH}) and generates the semigroup $H_t$  of linear self-adjoint contractions on $L^2(dw)$. The semigroup has the form
  \begin{align*}
  H_t f(\mathbf x)=\mathcal F^{-1}(e^{-t\|\xi\|^2}\mathcal Ff(\xi))(\mathbf x)=\int_{\mathbb R^N} h_t(\mathbf x,\mathbf y)f(\mathbf y)\, dw(\mathbf y),
  \end{align*}
  where the heat kernel 
  \begin{equation}\label{eq:heat_def}
      h_t(\mathbf x,\mathbf y)=\tau_{\mathbf x}h_t(-\mathbf y), \ \ h_t(\mathbf x)=\mathcal F^{-1} (e^{-t\|\xi\|^2})(\mathbf x)=c_k^{-1} (2t)^{-\mathbf N\slash 2}e^{-\| \mathbf x\|^2\slash (4t)}
  \end{equation}
  is a $C^\infty$-function of all variables $\mathbf x,\mathbf y \in \mathbb{R}^N$, $t>0$, and satisfies \begin{align*} 0<h_t(\mathbf x,\mathbf y)=h_t(\mathbf y,\mathbf x),
  \end{align*}
 \begin{equation} \label{eq:h_integral_1}
    \int_{\mathbb R^N} h_t(\mathbf x,\mathbf y)\, dw(\mathbf y)=1.
 \end{equation}

Let 
$$d(\mathbf x,\mathbf y)=\min_{\sigma\in G}\| \sigma(\mathbf x)-\mathbf y\|$$
be the distance of the orbit of $\mathbf x$ to the orbit of $\mathbf y$. Let us denote
\begin{equation}\label{eq:mathcal_G}
     \mathcal G_t(\mathbf x,\mathbf y)=(\max (w(B(\mathbf x,t)),w(B(\mathbf y, t))))^{-1}\exp\Big(-\frac {d(\mathbf x,\mathbf y)^2}{t}\Big).
 \end{equation}
  We shall need the following estimates for $h_t(\mathbf x,\mathbf y)$ - the proof can be found in~\cite[Theorem 4.1]{ADzH} and~\cite[Theorem 3.1]{DzH1}. 
\begin{theorem}\label{theorem:heat}
There are constants $C,c>0$ such that for all $\mathbf{x},\mathbf{y} \in \mathbb{R}^N$ and $t>0$ we have
\begin{equation}\label{eq:Gauss}
h_t(\mathbf{x},\mathbf{y}) \leq C\Big(1+\frac{\|\mathbf{x}-\mathbf{y}\|}{t}\Big)^{-2}\mathcal{G}_{t/c}(\mathbf{x},\mathbf{y}).
\end{equation}
\end{theorem}

Theorem~\ref{theorem:heat} imply the following Lemma (see~\cite[Corollary 3.5]{DzH1}).

\begin{lemma}\label{lem:nonradial_estimation}
Suppose that $\varphi \in C_c^{\infty}(\mathbb{R}^N)$ is radial and supported by the unit ball. Then there is $C>0$ such that for all $\mathbf{x},\mathbf{y}\in \mathbb{R}^{N}$ and $t>0$ we have
\begin{align*}
    &|\tau_{\mathbf x}\varphi(-\mathbf y)| \leq C \Big(1+\frac{\|\mathbf{x}-\mathbf{y}\|}{t}\Big)^{-2}(\max(w(B(\mathbf{x},t)),w(B(\mathbf{y},t))))^{-1}\chi_{[0,1]}(d(\mathbf{x},\mathbf{y})/t).
\end{align*}
\end{lemma}

\subsection{Dunkl-Schr\"odinger operator and semigroup}\label{sec:Schorodinger}
We present the main tools on Dunkl--Schr\"odinger operators, which are discussed in~\cite{AH} (see also~\cite{AH_2}) in details.
Let $V \geq 0$ be a measurable function such that $V \in L^2_{\rm{loc}}(dw)$. We consider the following operator on the Hilbert space $L^2(dw)$:
\begin{equation}\label{eq:Schrodinger_operator}
    \mathcal{L}=-\Delta+V
\end{equation}
with domain
\begin{align*}
    \mathcal{D}(\mathcal{L})=\{f \in L^2(dw)\,:\,\|\xi\|^2\mathcal{F}f(\xi) \in L^2(dw(\xi)) \text{ and }V(\mathbf{x})f(\mathbf{x}) \in L^2(dw(\mathbf{x}))\}.
\end{align*}
We call this operator the \textit{Dunkl-Schr\"odinger operator}. Let us define the quadratic form
\begin{equation}
    \mathbf{Q}(f,g)=\sum_{j=1}^{N}\int_{\mathbb{R}^N}T_jf(\mathbf{x})\overline{T_jg(\mathbf{x})}\,dw(\mathbf{x})+\int_{\mathbb{R}^N}V(\mathbf{x})f(\mathbf{x})\overline{g(\mathbf{x})}\,dw(\mathbf{x})
\end{equation}
with domain
\begin{align*}
    \mathcal{D}(\mathbf{Q})=\left\{f \in L^2(dw)\;:\; \left(\sum_{j=1}^{N}|T_jf|^2\right)^{1/2},V^{1/2}f \in L^2(dw)\right\} .
\end{align*}
The quadratic form is densely defined and closed (see~\cite[Lemma 4.1]{AH}), so there exists a unique positive self-adjoint operator $L$ such that
\begin{align*}
    \langle Lf,f\rangle=\mathbf{Q}(f,f) \text{ for all }f \in \mathcal{D}(L), 
\end{align*}
moreover,
\begin{align*}
    \mathcal{D}(L^{1/2})=\mathcal{D}(\mathbf{Q}) \text{ and }\mathbf{Q}(f,f)=\|L^{1/2}f\|_{L^2(dw)},
\end{align*}
where $L^{1/2}$ is a unique self-adjoint operator such that $(L^{1/2})^2=L$. It was proved in~\cite[Theorem 4.6]{AH}, that $\mathcal{L}$ is essentially self-adjoint on $C^{\infty}_{0}(\mathbb{R}^N)$ and $L$ is its closure. Consequently, $L$ generates the semigroup of self-adjoint contractions on $L^2(dw)$. The semigroup has the form (see~\cite[Theorem 4.8]{AH})
\begin{equation}
    K_tf(\mathbf{x})=\int_{\mathbb{R}^N}k_t(\mathbf{x},\mathbf{y})\,dw(\mathbf{y}),
\end{equation}
where $k_t(\mathbf{x},\mathbf{y})$ is the integral kernel which satisfies
\begin{equation}\label{eq:kernels_compare}
   0 \leq k_t(\mathbf{x},\mathbf{y}) \leq h_t(\mathbf{x},\mathbf{y}). 
\end{equation}

\part{Fefferman--Phong inequality}\label{part:Fefferman-Phong}

\section{Potential satisfying reverse Hölder inequality}
In this part, we assume that $q > \max(1,\frac{\mathbf{N}}{2})$ and $V$ belongs to the reverse H\"older class $\text{RH}^{q}(dw)$, that is, there is a constant $C_{\text{RH}}>0$ such that
\begin{equation}\label{eq:reverse_Holder}
    \left(\frac{1}{w(B)}\int_{B}V(\mathbf{x})^q\,dw(\mathbf{x})\right)^{1/q} \leq C_{\text{RH}}\frac{1}{w(B)}\int_{B}V(\mathbf{x})\,dw(\mathbf{x}) \text{ for every ball }B.
\end{equation}
For any Lebesque measurable set $A$ we define
\begin{equation}\label{eq:mu}
    \mu(A)=\int_{A}V(\mathbf{x})\,dw(\mathbf{x}).
\end{equation}

Our goal is to study the properties of the measure $\mu$ defined above. The proofs of the results in this section are standard and they are based on~\cite[Chapter 7]{Grafakos}.
\begin{lemma}
For all balls $B \subset \mathbb{R}^N$ and measurable sets $E \subseteq B$ we have
\begin{equation}\label{eq:q'_compare}
    \frac{\mu(E)}{\mu(B)} \leq C_{\text{RH}}\left(\frac{w(E)}{w(B)}\right)^{1/q'}, 
\end{equation}
where, here and subsequently, $\frac{1}{q}+\frac{1}{q'}=1$.
\end{lemma}

\begin{proof}
Applying H\"older's inequality, then the reverse H\"older inequality~\eqref{eq:reverse_Holder}, we get
\begin{align*}
    \mu(E)=\int_{\mathbb{R}^N}\chi_{E}(\mathbf{x})V(\mathbf{x})\,dw(\mathbf{x}) \leq w(E)^{1/q'}\left(\int_{B}V(\mathbf{x})^q\,dw(\mathbf{x})\right)^{1/q} \leq C_{\text{RH}}\left(\frac{w(E)}{w(B)}\right)^{1/q'}\mu(B).
\end{align*}
\end{proof}

\begin{lemma}\label{lem:measure_cont}
Let $\varepsilon>0$. There is a constant $0<\gamma<1$ such that for all $\mathbf{x} \in \mathbb{R}^N$ and $r>0$ we have
\begin{align*}
    1-\frac{w(B(\mathbf{x},\gamma r))}{w(B(\mathbf{x},r))}=\frac{w(B(\mathbf{x}, r) \setminus B(\mathbf{x},\gamma r))}{w(B(\mathbf{x},r))} < \varepsilon.
\end{align*}
\end{lemma}

\begin{proof}
Thanks to~\eqref{eq:measure_formula} we obtain
\begin{align*}
    w(B(\mathbf{x}, r) \setminus B(\mathbf{x},\gamma r))&=\int_{B(\mathbf{x},r) \setminus B(\mathbf{x},\gamma r)} \prod_{\alpha \in R}|\langle \mathbf{y},\alpha \rangle|^{k(\alpha)}\,d\mathbf{y} \\&\leq \int_{B(\mathbf{x},r) \setminus B(\mathbf{x},\gamma r)} \prod_{\alpha \in R}(|\langle \mathbf{y}-\mathbf{x},\alpha \rangle|+|\langle \mathbf{x},\alpha \rangle|)^{k(\alpha)}\,d\mathbf{y}.
\end{align*}
For all $\mathbf{y} \in B(\mathbf{x}, r) $ we have $|\langle \mathbf{y}-\mathbf{x},\alpha\rangle| \leq \sqrt{2}r$, so
\begin{align*}
    w(B(\mathbf{x}, r) \setminus B(\mathbf{x},\gamma r)) &\leq 2^{\mathbf{N}/2} \int_{B(\mathbf{x},r) \setminus B(\mathbf{x},\gamma r)} \prod_{\alpha \in R}(|\langle \mathbf{x},\alpha \rangle|+r)^{k(\alpha)}\,d\mathbf{y}\\&=v_{N}2^{\mathbf{N}/2}(r^{N}-\gamma^{N}r^{N})\prod_{\alpha \in R}(|\langle \mathbf{x},\alpha \rangle|+r)^{k(\alpha)},
\end{align*}
where $v_N$ is the Euclidean measure of the unit $N$-dimensional ball. Consequently, thanks to~\eqref{eq:balls_asymp}, we have
\begin{align*}
    \frac{w(B(\mathbf{x}, r) \setminus B(\mathbf{x},\gamma r))}{w(B(\mathbf{x},r))} \leq C(1-\gamma^{N}),
\end{align*}
where the constant $C>0$ is independent of $\mathbf{x}$ and $r$. The claim follows easily.
\end{proof}

\begin{lemma}\label{lem:mu_doubling}
The measure $\mu$ defined in~\eqref{eq:mu} is doubling, i.e. there is a constant $C_{\mu}>0$ such that for all $\mathbf{x} \in \mathbb{R}$ and $r>0$ we have
\begin{equation*}
    \mu(B(\mathbf{x},2r)) \leq C_{\mu}\mu(B(\mathbf{x},r)).
\end{equation*}
\end{lemma}

\begin{proof}
Let $0<\gamma<1$. Setting $B=B(\mathbf{x},r)$ and $E=B(\mathbf{x}, r) \setminus B(\mathbf{x},\gamma r)$ in~\eqref{eq:q'_compare}, we get
\begin{equation}\label{eq:complementary}
    1-\frac{\mu(B(\mathbf{x},\gamma r))}{\mu(B(\mathbf{x}, r))} \leq C_{\text{RH}}\left(1-\frac{w(B(\mathbf{x},\gamma r))}{w(B(\mathbf{x},r))}\right)^{1/q'}.
\end{equation}
Thanks to Lemma~\ref{lem:measure_cont} for $1-\gamma$ small enough we have
\begin{align*}
    C_{\text{RH}}\left(1-\frac{w(B(\mathbf{x},\gamma r))}{w(B(\mathbf{x},r))}\right)^{1/q'} <1/2,
\end{align*}
consequently,
\begin{equation}\label{eq:gamma_doubling}
    \mu(B(\mathbf{x},r)) \leq 2\mu(B(\mathbf{x},\gamma r)).
\end{equation}
There is $n \in \mathbb{N}$ such that $\gamma^{n}<1/2$. Applying~\eqref{eq:gamma_doubling} $n$ times we get the claim.
\end{proof}

As the consequence of the doubling property of $\mu$, we obtain the following corollary.
\begin{corollary}
There is a constant $C_{\text{RH}}'>0$ such that for all cubes $Q \subset \mathbb{R}^N$ and measurable sets $E \subseteq Q$ we have
\begin{equation}\label{eq:reverse_Holder_cubes}
    \left(\frac{1}{w(Q)}\int_{Q}V(\mathbf{x})^q\,dw(\mathbf{x})\right)^{1/q} \leq C'_{\text{RH}} \frac{1}{w(Q)}\int_{Q}V(\mathbf{x})\,dw(\mathbf{x}),
\end{equation}
\begin{equation}\label{eq:compare_cubes}
    \frac{\mu(E)}{\mu(Q)} \leq C_{\text{RH}}'\left(\frac{w(E)}{w(Q)}\right)^{1/q'}. 
\end{equation}
\end{corollary}

\begin{lemma}
There are $0<\gamma,\delta<1$ such that for all cubes $Q \subset \mathbb{R}^N$ and measurable sets $E \subseteq Q$ the following implication is true:
\begin{equation}\label{eq:implication}
    \mu(E) < \gamma \mu(Q) \Rightarrow w(E) < \delta w(Q).
\end{equation}
\end{lemma}

\begin{proof}
Set $\gamma'>0$ small enough in order to have $\delta'=C'_{\text{RH}}(\gamma')^{1/q'}<1$, where $C'_{\text{RH}}$ is the constant in~\eqref{eq:compare_cubes}. Then by~\eqref{eq:compare_cubes} we have the implication
\begin{equation}\label{eq:implication_1}
    w(E) \leq  \gamma' w(Q) \Rightarrow \mu(E) \leq \delta' \mu(Q).
\end{equation}
Taking $Q \setminus E$ instead of $E$ in~\eqref{eq:implication_1} we get
\begin{equation}\label{eq:implication_2}
    w(E) \geq  (1-\gamma') w(Q) \Rightarrow \mu(E) \geq  (1-\delta') \mu(Q).
\end{equation}
Note that~\eqref{eq:implication_2} is equivalent to~\eqref{eq:implication} with $\gamma=1-\delta'$ and $\delta=1-\gamma'$.
\end{proof}

We will need the following classical result from theory of $A_p$ weights (see~\cite[Corollary 7.2.4]{Grafakos}).

\begin{proposition}\label{propo:Grafakos}
Let $v$ be the weight and let $\nu$ be a doubling measure on $\mathbb{R}^N$. Suppose that there are $0<\gamma,\delta<1$ such that
\begin{equation*}
    \nu(E)<\gamma \nu(Q) \Rightarrow \int_{E}v(\mathbf{x})\,d\nu(\mathbf{x}) < \delta \int_{Q}v(\mathbf{x})\,d\nu(\mathbf{x}),
\end{equation*}
whenever $E$ is a $\nu$-measurable subset of a cube $Q$. Then there are constants $C,\eta>0$ such that for every cube $Q$ in $\mathbb{R}^N$ we have
\begin{equation}\label{eq:reverse_holder_different_weight}
    \left(\frac{1}{\nu(Q)}\int_{Q}v^{1+\eta}(\mathbf{x})\,d\nu(\mathbf{x})\right)^{1/(1+\eta)} \leq C \frac{1}{\nu(Q)}\int_{Q}v(\mathbf{x})\,d\nu(\mathbf{x}).
\end{equation}
\end{proposition}

\begin{proposition}
There is a constant $C>0$ and $p>1$ such that for every cube $Q$ in $\mathbb{R}^N$ we have
\begin{equation}\label{eq:A_p}
    \left(\frac{1}{w(Q)}\int_{Q} V(\mathbf{x})\,dw(\mathbf{x})\right)\left(\frac{1}{w(Q)}\int_{Q}V^{-\frac{1}{p-1}}(\mathbf{x})\,dw(\mathbf{x})\right)^{p-1} \leq C.
\end{equation}
\end{proposition}

\begin{proof}
Note that~\eqref{eq:implication} is equivalent to
\begin{equation}
    \mu(E) < \gamma \mu(Q) \Rightarrow \int_{E}V^{-1}(\mathbf{x})\,d\mu(\mathbf{x}) <  \delta \int_{Q}V^{-1}(\mathbf{x})\,d\mu(\mathbf{x}).
\end{equation}
Hence, applying Proposition~\ref{propo:Grafakos} to $v=V^{-1}$ and $\nu=\mu$ (the assumption that $\nu$ is doubling is satisfied thanks to Lemma~\ref{lem:mu_doubling}) we get that there are $C,\eta>0$ such that
\begin{equation}\label{eq:reverse_app}
    \left(\frac{1}{\mu(Q)}\int_{Q}V(\mathbf{x})^{-1-\eta}V(\mathbf{x})\,dw(\mathbf{x})\right)^{1/(1+\eta)} \leq C \frac{1}{\mu(Q)}\int_{Q}V(\mathbf{x})^{-1}V(\mathbf{x})\,dw(\mathbf{x})=C\frac{w(Q)}{\mu(Q)}.
\end{equation}
Finally, it can be checked that~\eqref{eq:reverse_app} is equivalent to~\eqref{eq:A_p} with $p=1+\frac{1}{\eta}$.
\end{proof}

The reverse H\"older inequality~\eqref{eq:reverse_Holder} has the following consequence (see~\cite[Lemma 1.2]{Shen}), which will be used in the next section many times.

\begin{lemma}\label{lem:Rr}
There is a constant $C>0$ such that for all $\mathbf{x} \in \mathbb{R}^N$ and $0<r_1<r_2<\infty$ we have
\begin{align*}
    \frac{r_1^2}{w(B(\mathbf{x},r_1))}\int_{B(\mathbf{x},r_1)}V(\mathbf{y})\,dw(\mathbf{y}) \leq C \left(\frac{r_1}{r_2}\right)^{2-\mathbf{N}/q} \frac{r_2^2}{w(B(\mathbf{x},r_2))}\int_{B(\mathbf{x},r_2)}V(\mathbf{y})\,dw(\mathbf{y}). 
\end{align*}
\end{lemma}

\begin{proof}
Thanks to H\"older's inequality and the reverse H\"older inequality~\eqref{eq:reverse_Holder}, we get
\begin{align*}
    \frac{1}{w(B(\mathbf{x},r_1))}\int_{B(\mathbf{x},r_1)}V(\mathbf{y})\,dw(\mathbf{y}) &\leq \left(\frac{1}{w(B(\mathbf{x},r_1))}\int_{B(\mathbf{x},r_1)}V(\mathbf{y})^q\,dw(\mathbf{y})\right)^{1/q} \\& \leq  \frac{w(B(\mathbf{x},r_2))^{1/q}}{w(B(\mathbf{x},r_1))^{1/q}}\left(\frac{1}{w(B(\mathbf{x},r_2))}\int_{B(\mathbf{x},r_2)}V(\mathbf{y})^q\,dw(\mathbf{y})\right)^{1/q} \\& \leq C_{\text{RH}}\frac{w(B(\mathbf{x},r_2))^{1/q}}{w(B(\mathbf{x},r_1))^{1/q}}\frac{1}{w(B(\mathbf{x},r_2))}\int_{B(\mathbf{x},r_2)}V(\mathbf{y})\,dw(\mathbf{y}).
\end{align*}
Finally, the claim follows by~\eqref{eq:growth}.
\end{proof}

\section{The auxiliary function \texorpdfstring{$m(\mathbf{x})$}{m(x)}}
\subsection{Definition and growth properties of \texorpdfstring{$m(\mathbf{x})$}{m(x)}}
For $\mathbf{x} \in \mathbb{R}^N$ we define (see~\cite[Definition 1.3]{Shen}):
\begin{equation}\label{eq:m}
    \frac{1}{m(\mathbf{x})}=\sup\left\{r>0\,:\; \frac{r^2}{w(B(\mathbf{x},r))}\int_{B(\mathbf{x},r)}V(\mathbf{y})\,dw(\mathbf{y}) \leq 1 \right\}.
\end{equation}
Thanks to Lemma~\ref{lem:Rr}, for all $\mathbf{x} \in \mathbb{R}^N$ (and $V \not\equiv 0$) we have
\begin{equation}\label{eq:well-define}
    \lim_{r \to 0} \frac{r^2}{w(B(\mathbf{x},r))}\int_{B(\mathbf{x},r)}V(\mathbf{y})\,dw(\mathbf{y})=0,\, \ \ \lim_{r \to +\infty} \frac{r^2}{w(B(\mathbf{x},r))}\int_{B(\mathbf{x},r)}V(\mathbf{y})\,dw(\mathbf{y})=+\infty,
\end{equation}
so the function $m$ is well-defined. The next lemma is an adaptation of~\cite[Lemma 1.4]{Shen}.

\begin{lemma}\label{lem:m_growth}
There are constants $C,\kappa>0$ such that for all $\mathbf{x},\mathbf{y} \in \mathbb{R}^N$ we have
\begin{equation}\label{eq:Shen_A}
    C^{-1}m(\mathbf{y}) \leq m(\mathbf{x}) \leq Cm(\mathbf{y}) \text { if }\|\mathbf{x}-\mathbf{y}\|<m(\mathbf{x})^{-1},
\end{equation}
\begin{equation}\label{eq:Shen_B}
    m(\mathbf{y}) \leq Cm(\mathbf{x})(1+\|\mathbf{x}-\mathbf{y}\|m(\mathbf{x}))^{\kappa},
\end{equation}
\begin{equation}\label{eq:Shen_C}
    m(\mathbf{y}) \geq C^{-1}m(\mathbf{x})(1+m(\mathbf{x})\|\mathbf{x}-\mathbf{y}\|)^{-\frac{\kappa}{1+\kappa}}.
\end{equation}
\end{lemma}

\begin{proof}[Proof of~\eqref{eq:Shen_A}]
By the doubling property of $w$ and $\mu$ we have $w(B(\mathbf{x},r)) \sim w(B(\mathbf{y},r))$ and $\mu(B(\mathbf{x},r)) \sim \mu(B(\mathbf{y},r))$ if $r \geq \|\mathbf{x}-\mathbf{y}\|$. So, by Lemma~\ref{lem:Rr}, for any $r<m(\mathbf{x})^{-1}$ we have 
\begin{equation}\label{eq:Shen_A_comp}
\begin{split}
    &\frac{r^2}{w(B(\mathbf{y},r))}\int_{B(\mathbf{y},r)}V(\mathbf{z})\,dw(\mathbf{z}) \leq C \left(\frac{r}{m(\mathbf{x})^{-1}}\right)^{2-\frac{\mathbf{N}}{q}} \frac{m(\mathbf{x})^{-2}}{w(B(\mathbf{y},m(\mathbf{x})^{-1}))}\int_{B(\mathbf{y},m(\mathbf{x})^{-1})}V(\mathbf{z})\,dw(\mathbf{z})\\& \leq C' \left(\frac{r}{m(\mathbf{x})^{-1}}\right)^{2-\frac{\mathbf{N}}{q}} \frac{m(\mathbf{x})^{-2}}{w(B(\mathbf{x},m(\mathbf{x})^{-1}))}\int_{B(\mathbf{x},m(\mathbf{x})^{-1})}V(\mathbf{z})\,dw(\mathbf{z}) \leq C'\left(\frac{r}{m(\mathbf{x})^{-1}}\right)^{2-\frac{\mathbf{N}}{q}},
\end{split}
\end{equation}
where in the last inequality we have used the definition of $m$. Note that~\eqref{eq:Shen_A_comp} implies that for 
\begin{align*}
    r<\min\left(1,(2C')^{-\frac{1}{2-\frac{\mathbf{N}}{q}}}\right)m(\mathbf{x})^{-1}
\end{align*}
we get
\begin{align*}
     &\frac{r^2}{w(B(\mathbf{y},r))}\int_{B(\mathbf{y},r)}V(\mathbf{z})\,dw(\mathbf{z}) \leq \frac{1}{2},
\end{align*}
so the inequality $m(\mathbf{y}) \leq Cm(\mathbf{x})$ follows. Now we turn to the proof of $m(\mathbf{x}) \leq Cm(\mathbf{y})$. For $r>2m(\mathbf{x})^{-1}$, thanks to the doubling property of $\mu$ and $w$, then Lemma~\ref{lem:Rr}, we write
\begin{align*}
    &\frac{r^2}{w(B(\mathbf{y},r))}\int_{B(\mathbf{y},r)}V(\mathbf{z})\,dw(\mathbf{z}) \geq C \frac{r^2}{w(B(\mathbf{x},r))}\int_{B(\mathbf{x},r)}V(\mathbf{z})\,dw(\mathbf{z})\\& \geq C'\left(\frac{r}{m(\mathbf{x})^{-1}}\right)^{2-\frac{\mathbf{N}}{q}}\frac{(2m(\mathbf{x})^{-1})^{2}}{w(B(\mathbf{x},2m(\mathbf{x})^{-1}))}\int_{B(\mathbf{x},2m(\mathbf{x})^{-1})}V(\mathbf{z})\,dw(\mathbf{z}) \geq C'\left(\frac{r}{m(\mathbf{x})^{-1}}\right)^{2-\frac{\mathbf{N}}{q}}.
\end{align*}
where in the last inequality we have used the definition of $m(\mathbf{x})$. Taking 
\begin{align*}
     r>\max\left(2,(C'/2)^{-\frac{1}{2-\frac{\mathbf{N}}{q}}}\right)m(\mathbf{x})^{-1}
\end{align*}
we have
\begin{align*}
    &\frac{r^2}{w(B(\mathbf{y},r))}\int_{B(\mathbf{y},r)}V(\mathbf{z})\,dw(\mathbf{z}) \geq 2,
\end{align*}
so, thanks to definition of $m$ (see~\eqref{eq:m}), we are done. 
\end{proof}

\begin{proof}[Proof of~\eqref{eq:Shen_B}]
We may assume $\|\mathbf{x}-\mathbf{y}\|m(\mathbf{x}) \geq 1$, otherwise the claim follows by~\eqref{eq:Shen_A}. Let $r=m(\mathbf{x})^{-1}$ and let $j \geq 1$, $j \in \mathbb{Z}$, be such that
\begin{align*}
    2^{j-1}r < \|\mathbf{x}-\mathbf{y}\| \leq 2^{j}r.
\end{align*}
Let $0<r_1<r$. Thanks to Lemma~\ref{lem:Rr}, then the doubling property of $\mu$ and $w$ together with~\eqref{eq:growth}, we have
\begin{align*}
    &\frac{r_1^2}{w(B(\mathbf{y},r_1))}\int_{B(\mathbf{y},r_1)}V(\mathbf{z})\,dw(\mathbf{z}) \leq C\left(\frac{r_1}{\|\mathbf{x}-\mathbf{y}\|}\right)^{2-\frac{\mathbf{N}}{q}}\frac{\|\mathbf{x}-\mathbf{y}\|^2}{w(B(\mathbf{y},\|\mathbf{x}-\mathbf{y}\|))}\int\limits_{B(\mathbf{y},\|\mathbf{x}-\mathbf{y}\|)}V(\mathbf{z})\,dw(\mathbf{z}) \\&\leq C\left(\frac{r_1}{\|\mathbf{x}-\mathbf{y}\|}\right)^{2-\frac{\mathbf{N}}{q}}\frac{\|\mathbf{x}-\mathbf{y}\|^2}{w(B(\mathbf{x},\|\mathbf{x}-\mathbf{y}\|))}\int_{B(\mathbf{x},\|\mathbf{x}-\mathbf{y}\|)}V(\mathbf{z})\,dw(\mathbf{z}) \\&\leq C \left(\frac{r_1}{2^{j}r}\right)^{2-\frac{\mathbf{N}}{q}} 2^{-jN}C_{\mu}^{j}\frac{2^{2j}r^2}{w(B(\mathbf{x},r))}\int_{B(\mathbf{x},r)}V(\mathbf{z})\,dw(\mathbf{z}) \leq C\left(\frac{r_1}{2^{j}r}\right)^{2-\frac{\mathbf{N}}{q}} 2^{j(2-N)}C_{\mu}^{j},
\end{align*}
where $C_\mu$ is the doubling constant for $ \mu $ (see Lemma~\ref{lem:mu_doubling}) and we have used~\eqref{eq:growth} and the definition of $m$ in the last line. Therefore, there is a constant $C_1>1$ independent of $\mathbf{x},\mathbf{y} \in \mathbb{R}^N$ and $r>r_1>0$ such that if $r_1 \leq rC_1^{-j}$, then
\begin{align*}
    \frac{r_1^2}{w(B(\mathbf{y},r_1))}\int_{B(\mathbf{y},r_1)}V(\mathbf{z})\,dw(\mathbf{z}) \leq C\left(\frac{r_1}{2^{j}r}\right)^{2-\frac{\mathbf{N}}{q}} 2^{-jN}C_{\mu}^{j}2^{2j} \leq \frac{1}{2}.
\end{align*}
Consequently, by the definition of $m(\mathbf{y})$ we have
\begin{align*}
    \frac{1}{m(\mathbf{y})} \geq rC_1^{-j}=\frac{1}{m(\mathbf{x})}C_1^{-j},
\end{align*}
which lead us to
\begin{align*}
    m(\mathbf{y}) \leq m(\mathbf{x})C_1^{j} \leq Cm(\mathbf{x})(1+m(\mathbf{x})\|\mathbf{x}-\mathbf{y}\|)^{\kappa},
\end{align*}
where $\kappa=\log_2 C_1$.
\end{proof}

\begin{proof}[Proof of~\eqref{eq:Shen_C}]
We may assume that $\|\mathbf{x}-\mathbf{y}\| \geq m(\mathbf{y})^{-1}$, otherwise the claim follows by~\eqref{eq:Shen_A}. By~\eqref{eq:Shen_B} we have
\begin{align*}
    m(\mathbf{x}) \leq C m(\mathbf{y})\left(1+\|\mathbf{x}-\mathbf{y}\|m(\mathbf{y})\right)^{\kappa} \leq Cm(\mathbf{y})^{1+\kappa}\|\mathbf{x}-\mathbf{y}\|^{\kappa}.
\end{align*}
Thus,
\begin{align*}
    m(\mathbf{y}) \geq C'\frac{m(\mathbf{x})^{1/(1+\kappa)}}{\|\mathbf{x}-\mathbf{y}\|^{\kappa/(1+\kappa)}} \geq C\frac{m(\mathbf{x})}{(1+m(\mathbf{x})\|\mathbf{x}-\mathbf{y}\|)^{\kappa/(1+\kappa)}},
\end{align*}
so the proof is complete.
\end{proof}

\subsection{Associated collection of cubes \texorpdfstring{$\mathcal{Q}$}{Q}}

For a cube $Q \subset \mathbb{R}^N$, here and subsequently, let $d(Q)$ denote the side-length of cube $Q$. We denote by $Q^{*}$ the cube with the same center as $Q$ such that $d(Q^{*})=2d(Q)$. We define a collection of dyadic cubes $\mathcal{Q}$ associated with the potential $V$ by the following stopping-time condition:
\begin{equation}\label{eq:stopping_time}
    Q \in \mathcal{Q} \iff Q \text{ is the maximal dyadic cube for which }\frac{d(Q)^2}{w(Q)}\int_{Q} V(\mathbf{y})\,dw(\mathbf{y}) \leq 1.
\end{equation}

Thanks to the doubling property of $w$ and $\mu$ together with~\eqref{eq:well-define} we see that the collection $\mathcal{Q}$ is well-defined and it forms a covering of $\mathbb{R}^N$ by disjoint dyadic cubes. We list below simple facts about the collection $\mathcal{Q}$, which are consequences of properties of $w$, $\mu$ and $m(\mathbf{x})$.
\begin{fact}
There is a constant $C>0$ such that for any $Q \in \mathcal{Q}$ we have
\begin{equation}\label{eq:stopping_time_opposite}
    C^{-1} \leq \frac{d(Q)^2}{w(Q)}\int_{Q}V(\mathbf{x})\,dw(\mathbf{x}).
\end{equation}
\end{fact}

\begin{proof}
It is an easy consequence of the doubling property of $\mu$. Namely, let $\widetilde{Q}$ be the parent of cube $Q \in \mathcal{Q}$. As the consequence of the stopping-time condition~\eqref{eq:stopping_time}, we get
\begin{align*}
    1<\frac{d(\widetilde{Q})^2}{w(\widetilde{Q})}\int_{\widetilde{Q}}V(\mathbf{x})\,dw(\mathbf{x}) \leq \frac{(2d(Q))^2}{w(Q)}\int_{\widetilde{Q}}V(\mathbf{x})\,dw(\mathbf{x}) \leq C \frac{d(Q)^2}{w(Q)}\int_{Q}V(\mathbf{x})\,dw(\mathbf{x}).
\end{align*}
\end{proof}

\begin{fact}\label{fact:m_d_compare}
There is a constant $C>0$ such that for any $Q \in \mathcal{Q}$ and $\mathbf{x} \in Q^{****}$ we have
\begin{equation}\label{eq:m_on_cube}
    C^{-1}d(Q)^{-1} \leq m(\mathbf{x}) \leq Cd(Q)^{-1}.
\end{equation}
\end{fact}

\begin{proof}
The proof is essentially the same as the proof of~\eqref{eq:Shen_A}. We provide details. Note that $Q^{****} \subseteq B(\mathbf{x},10^2 d(Q))$ for $\mathbf{x} \in Q^{****}$.  Therefore, by the doubling property of $\mu$ and $w$ together with~\eqref{eq:stopping_time_opposite} we have
\begin{align*}
    &(C')^{-1} \leq C^{-1}\frac{d(Q)^2}{w(Q)}\int_{Q}V(\mathbf{y})\,dw(\mathbf{y}) \leq \frac{(10^2d(Q))^2}{w(B(\mathbf{x},10^2d(Q)))}\int_{B(\mathbf{x},10^2d(Q))}V(\mathbf{y})\,dw(\mathbf{y})  \\& \leq C\frac{d(Q)^2}{w(Q)}\int_{Q}V(\mathbf{y})\,dw(\mathbf{y}) \leq C'.
\end{align*}
Consequently, for $r<10^2d(Q)$, by Lemma~\ref{lem:Rr} with $r_1=r$ and $r_2=10^{2}d(Q)$, we have
\begin{align*}
    \frac{r^2}{w(B(\mathbf{x},r))}\int_{B(\mathbf{x},r)}V(\mathbf{y})\,dw(\mathbf{y}) \leq C\left(\frac{r}{d(Q)}\right)^{2-\frac{\mathbf{N}}{q}}.
\end{align*}
By the same argument as in the proof of~\eqref{eq:Shen_A} we have $m(\mathbf{x}) \leq Cd(Q)^{-1}$. Similarly, for $r>10^2d(Q)$, we have
\begin{align*}
    \frac{r^2}{w(B(\mathbf{x},r))}\int_{B(\mathbf{x},r)}V(\mathbf{y})\,dw(\mathbf{y}) \geq C\left(\frac{r}{d(Q)}\right)^{2-\frac{\mathbf{N}}{q}},
\end{align*}
so repeating the argument from the proof of~\eqref{eq:Shen_A} we have $Cm(\mathbf{x}) \geq d(Q)^{-1}$.
\end{proof}

Lemma~\ref{lem:m_growth} together with Fact~\ref{fact:m_d_compare} imply the following claim.

\begin{fact}\label{fact:F}
The collection $\mathcal{Q}$ satisfies~\eqref{eq:finite_overlap} from Section~\ref{sec:Hardy_pre}.
\end{fact}

\section{Fefferman--Phong inequality}\label{sec:Fefferman--Phong}

The goal of this section is the prove Fefferman--Phong inequality in the rational Dunkl setting. This result is crucial in the proof of condition~\eqref{eq:D} (see Section~\ref{sec:Hardy_pre}) for potential satisfying~\eqref{eq:reverse_Holder}. The result for $k \equiv 0$ is due to C. Feffermann and D.H. Phong~\cite{Fefferman} (see also~\cite[Lemma 1.9]{Shen}). The proof is inspired by the proof from~\cite[Theorem 9.4]{DZ_Studia}.
\begin{theorem}[Fefferman--Phong type inequality]\label{theo:Fefferman-Phong}
There is a constant $C>0$ such that for all $f \in \mathcal{D}(\mathbf{Q})$ we have
\begin{equation}\label{eq:Fefferman-Phong}
    \int_{\mathbb{R}^N}|f(\mathbf{x})|^2m(\mathbf{x})^2\,dw(\mathbf{x}) \leq C \mathbf{Q}(f,f).
\end{equation}
\end{theorem}

We need some lemmas before providing the proof of Theorem~\ref{theo:Fefferman-Phong}.
\begin{lemma}\label{lem:E_set}
There are constants $C,\eta>0$ such that for all $Q \in \mathcal{Q}$ and $\varepsilon>0$ we have $w(E_{\varepsilon}) \leq C\varepsilon^{\eta}w(Q^{*})$, where
\begin{equation}\label{eq:E_def}
    E_{\varepsilon}=\{\mathbf{y} \in Q^{*}\;:\; V(\mathbf{y}) \leq \varepsilon d(Q)^{-2}\}.
\end{equation}
\end{lemma}

\begin{proof}
Let $p>1$ be the number from~\eqref{eq:A_p}. By the definition of $E_{\varepsilon}$ we write
\begin{equation}\label{eq:E_set_computation}
\begin{split}
    w(E_{\varepsilon})^{p-1} &= \left(\int_{E_{\varepsilon}}\,dw(\mathbf{y})\right)^{p-1} \leq \left(\int_{E_{\varepsilon}}\varepsilon^{1/(p-1)}d(Q)^{-2/(p-1)}V(\mathbf{y})^{-\frac{1}{p-1}}\,dw(\mathbf{y})\right)^{p-1}\\& \leq \varepsilon d(Q)^{-2}\left(\int_{Q^{*}}V(\mathbf{y})^{-\frac{1}{p-1}}\,dw(\mathbf{y})\right)^{p-1}.
\end{split}
\end{equation}
Thanks to~\eqref{eq:stopping_time_opposite} and the doubling property of $w$ we have
\begin{equation}\label{eq:stopping_time_opposite_app}
    d(Q)^{-2} \leq C\frac{1}{w(Q)}\int_{Q}V(\mathbf{y})\,dw(\mathbf{y}) \leq C'\frac{1}{w(Q^{*})}\int_{Q^{*}}V(\mathbf{y})\,dw(\mathbf{y}) .
\end{equation}
Consequently, applying~\eqref{eq:E_set_computation} and~\eqref{eq:stopping_time_opposite_app} together with~\eqref{eq:A_p} we get
\begin{align*}
    w(E_{\varepsilon})^{p-1} \leq C\varepsilon \left(\frac{1}{w(Q^{*})}\int_{Q^{*}}V(\mathbf{y})\,dw(\mathbf{y})\right)\left(\int_{Q^{*}}V(\mathbf{y})^{-\frac{1}{p-1}}\,dw(\mathbf{y})\right)^{p-1} \leq C'\varepsilon w(Q^{*})^{p-1}.
\end{align*}
\end{proof}

\begin{lemma}\label{lem:general_Leibniz}
For all  $j \in \{1,2,\ldots,N\}$, $g \in C^\infty_{c}(\mathbb{R}^N)$, and $f \in L^2(dw)$ such that its weak Dunkl derivative $T_jf$ is in $L^2(dw)$ we have $T_j(fg) \in L^2(dw)$. Moreover,
\begin{equation}\label{eq:general_Leibniz}
    T_j(fg)(\mathbf{x})=(T_jf)(\mathbf{x})g(\mathbf{x})+f(\mathbf{x})\partial_{j}g(\mathbf{x})+\sum_{\alpha \in R}\frac{k(\alpha)}{2} \alpha_j f(\sigma_{\alpha}(\mathbf{x}))\frac{g(\mathbf{x})-g(\sigma_{\alpha}(\mathbf{x}))}{\langle \mathbf{x}, \alpha \rangle}
\end{equation}
in $L^2(dw)$-sense.
\end{lemma}

\begin{proof}
It is a standard fact, but for the convenience of reader we provide the proof. Let us assume first that $f \in C^1(\mathbb{R}^N)$. By the definition of $T_j$ (see~\eqref{eq:T_def}) we have
\begin{equation}\label{eq:computation_C1}
\begin{split}
    &T_j(fg)(\mathbf{x})=\partial_{j}(fg)(\mathbf{x})+\sum_{\alpha \in R}\frac{k(\alpha)}{2}\alpha_j \frac{f(\mathbf{x})g(\mathbf{x})-f(\sigma_{\alpha}(\mathbf{x}))g(\sigma_{\alpha}(\mathbf{x}))}{\langle \mathbf{x}, \alpha \rangle}\\&=f(\mathbf{x})(\partial_{j}g)(\mathbf{x})+(\partial_{j}f)(\mathbf{x})g(\mathbf{x})+\sum_{\alpha \in R}\frac{k(\alpha)}{2} \alpha_j g(\mathbf{x})\frac{f(\mathbf{x})-f(\sigma_{\alpha}(\mathbf{x}))}{\langle \mathbf{x}, \alpha \rangle}\\&+\sum_{\alpha \in R}\frac{k(\alpha)}{2} \alpha_j f(\sigma_{\alpha}(\mathbf{x}))\frac{g(\mathbf{x})-g(\sigma_{\alpha}(\mathbf{x}))}{\langle \mathbf{x}, \alpha \rangle}\\&=f(\mathbf{x})\partial_{j}g(\mathbf{x})+(T_jf)(\mathbf{x})g(\mathbf{x})+\sum_{\alpha \in R}\frac{k(\alpha)}{2} \alpha_j f(\sigma_{\alpha}(\mathbf{x}))\frac{g(\mathbf{x})-g(\sigma_{\alpha}(\mathbf{x}))}{\langle \mathbf{x}, \alpha \rangle}.
\end{split}
\end{equation}
In order to obtain the  general case, let us take $\psi \in C^{\infty}_{c}(\mathbb{R}^N)$. By the definition of $T_j(fg)$ and~\eqref{eq:computation_C1} we have
\begin{align*}
    &\int_{\mathbb{R}^N} T_{j}(fg)(\mathbf{x})\psi(\mathbf{x})\,dw(\mathbf{x})=-\int_{\mathbb{R}^N} f(\mathbf{x})g(\mathbf{x})T_j\psi(\mathbf{x})\,dw(\mathbf{x})\\&=-\int_{\mathbb{R}^N}f(\mathbf{x})T_{j}(g\psi)(\mathbf{x})+\int_{\mathbb{R}^N}f(\mathbf{x})\partial_{j}g(\mathbf{x})\psi(\mathbf{x})\,dw(\mathbf{x})\\&+\sum_{\alpha \in R}\frac{k(\alpha)}{2}\alpha_j\int_{\mathbb{R}^N}f(\mathbf{x})\psi(\sigma_{\alpha}(\mathbf{x}))\frac{g(\mathbf{x})-g(\sigma_{\alpha}(\mathbf{x}))}{\langle \mathbf{x},\alpha \rangle}\,dw(\mathbf{x})\\&=\int_{\mathbb{R}^N}T_jf(\mathbf{x})g(\mathbf{x})\psi(\mathbf{x})\,dw(\mathbf{x})+\int_{\mathbb{R}^N}f(\mathbf{x})\partial_{j}g(\mathbf{x})\psi(\mathbf{x})\,dw(\mathbf{x})\\&+\sum_{\alpha \in R}\frac{k(\alpha)}{2}\alpha_j\int_{\mathbb{R}^N}f(\sigma_{\alpha}(\mathbf{x}))\frac{g(\mathbf{x})-g(\sigma_{\alpha}(\mathbf{x}))}{\langle \mathbf{x},\alpha \rangle}\psi(\mathbf{x})\,dw(\mathbf{x}).
\end{align*}
\end{proof}

Let $\{\phi_{Q}\}_{Q \in \mathcal{Q}}$ be a smooth resolution of identity associated with $\mathcal{Q}$, that means the collection of $C^{\infty}$-functions on $\mathbb{R}^N$ such that $\supp \phi_{Q} \subseteq Q^{*}$, $0 \leq \phi_Q(\mathbf{x}) \leq 1$,
\begin{equation}\label{eq:partition_derivative}
    |\partial^{\alpha}\phi_{Q}(\mathbf{x})| \leq C_{\alpha}d(Q)^{-|\alpha|} \text{ for all }\alpha \in \mathbb{N}_0^{N},
\end{equation}
and $\sum_{Q \in \mathcal{Q}}\phi_Q(\mathbf{x})=1$ for all $\mathbf{x} \in \mathbb{R}^N$. The existence of $\{\phi_{Q}\}_{Q \in \mathcal{Q}}$ is guaranteed by~\eqref{eq:finite_overlap} (see Fact~\ref{fact:F}).

\begin{lemma}\label{lem:difference_by_derivative}
There is a constant $C>0$ such that for all $\alpha \in R$, $Q \in \mathcal{Q}$, and $\mathbf{x} \in Q^{*}$ we have
\begin{align*}
    \left|\frac{\phi(\mathbf{x})-\phi(\sigma_{\alpha}(\mathbf{x}))}{\langle \mathbf{x}, \alpha \rangle}\right| \leq Cd(Q)^{-1}.
\end{align*}
\end{lemma}

\begin{proof}
This is the standard fact - we write
\begin{align*}
    \frac{\phi(\mathbf{x})-\phi(\sigma_{\alpha}(\mathbf{x}))}{\langle \mathbf{x}, \alpha \rangle}&=-\frac{1}{\langle \mathbf{x},\alpha \rangle}\int_0^{1}\frac{d}{dt}\phi\left(\mathbf{x}-2t\frac{\langle \mathbf{x},\alpha\rangle }{\|\alpha\|^2}\alpha\right)\,dt\\&=\int_{0}^{1}\left\langle (\nabla_{\mathbf{x}}\phi)\left(\mathbf{x}-2t\frac{\langle \mathbf{x},\alpha\rangle }{\|\alpha\|^2}\alpha \right),\alpha \right\rangle \,dt,
\end{align*}
so the claim is a consequence of~\eqref{eq:partition_derivative}.
\end{proof}

\begin{lemma}\label{lem:gradient}
There is a constant $C>0$ such that for all $j\in \{1,\ldots,N\}$, $f \in L^2(dw)$ such that its weak Dunkl derivative $T_jf$ is in $L^2(dw)$, and $Q \in \mathcal{Q}$ we have
\begin{align*}
    \|T_j(f\phi_{Q})\|_{L^2(dw)} \leq C\left(\left(\int_{Q^{*}}|T_jf(\mathbf{x})|^2\,dw(\mathbf{x})\right)^{1/2}+\left(\int_{\mathcal{O}(Q^{*})}|f(\mathbf{x})|^2m(\mathbf{x})^2\,dw(\mathbf{x})\right)^{1/2}\right)
\end{align*}
(let us remind that $\mathcal{O}(Q^{*})$ denotes the orbit of cube $Q^{*}$, see~\eqref{eq:orbit_of_A}).
\end{lemma}

\begin{proof}
By Lemma~\ref{lem:general_Leibniz} we have
\begin{align*}
    \|T_j(f\phi_{Q})\|_{L^2(dw)} &\leq \|(T_jf)(\mathbf{x})\phi_{Q}(\mathbf{x})\|_{L^2(dw(\mathbf{x}))}+\|f(\mathbf{x})(\partial_{j}\phi_{Q})(\mathbf{x})\|_{L^2(dw(\mathbf{x}))}\\&+C\sum_{\alpha \in R}\|f(\sigma_{\alpha}(\mathbf{x}))\frac{\phi_{Q}(\mathbf{x})-\phi_{Q}(\sigma_{\alpha}(\mathbf{x}))}{\langle \mathbf{x},\alpha\rangle}\|_{L^2(dw(\mathbf{x}))}.
\end{align*}
Thanks to the property that $\supp \phi_{Q} \subseteq Q^{*}$,~\eqref{eq:partition_derivative}, and Fact~\ref{fact:m_d_compare} we have
\begin{align*}
    \|(T_jf)(\mathbf{x})\phi_{Q}(\mathbf{x})\|_{L^2(dw(\mathbf{x}))} \leq C\left(\int_{Q^{*}}|T_jf(\mathbf{x})|^2\,dw(\mathbf{x})\right)^{1/2},
\end{align*}
\begin{align*}
    \|f(\mathbf{x})(\partial_{j}\phi_{Q})(\mathbf{x})\|_{L^2(dw(\mathbf{x}))} \leq C\left(\int_{Q^{*}}|f(\mathbf{x})|^2 m(\mathbf{x})^2\,dw(\mathbf{x})\right)^{1/2}.
\end{align*}
Therefore, it is enough to estimate
\begin{align*}
    &\int_{\mathcal{O}(Q^{*})}\left|f(\sigma_{\alpha}(\mathbf{x}))\frac{\phi_{Q}(\mathbf{x})-\phi_{Q}(\sigma_{\alpha}(\mathbf{x}))}{\langle \mathbf{x},\alpha\rangle}\right|^2\,dw(\mathbf{x})\\&=\int_{\mathcal{O}(Q^{*}) \cap \{\mathbf{x}\;:\;\sqrt{2}|\langle \mathbf{x},\alpha \rangle|\leq m(\mathbf{x})^{-1}\}}\ldots+\int_{\mathcal{O}(Q^{*}) \cap \{\mathbf{x}\;:\;\sqrt{2}|\langle \mathbf{x},\alpha \rangle| > m(\mathbf{x})^{-1}\}}\ldots=:I_1+I_2
\end{align*}
for fixed $\alpha \in R$. We consider $I_1$ first. Let us denote 
\begin{equation}\label{eq:non-zero}
    E=\mathcal{O}(Q^{*}) \cap \{\mathbf{x}\;:\;\sqrt{2}|\langle \mathbf{x},\alpha \rangle|\leq m(\mathbf{x})^{-1}\} \cap\{\mathbf{x}\;:\;\frac{\phi_{Q}(\mathbf{x})-\phi_{Q}(\sigma_{\alpha}(\mathbf{x}))}{\langle \mathbf{x},\alpha\rangle} \neq 0\}.
\end{equation}
If $\mathbf{x} \in E$, then either $\mathbf{x} \in Q^{*}$ or $\sigma_{\alpha}(\mathbf{x}) \in Q^{*}$, so, by Fact~\ref{fact:m_d_compare}, $d(Q)^{-1} \leq Cm(\mathbf{x})$ or $d(Q)^{-1} \leq Cm(\sigma_{\alpha}(\mathbf{x}))$ respectively. Note that
\begin{align*}
    \mathcal{O}(Q^{*}) \cap \{\mathbf{x}\;:\;\sqrt{2}|\langle \mathbf{x},\alpha \rangle|\leq m(\mathbf{x})^{-1}\}=\mathcal{O}(Q^{*}) \cap \{\mathbf{x}\;:\;\| \mathbf{x}-\sigma_{\alpha}(\mathbf{x}) \|\leq m(\mathbf{x})^{-1}\}, 
\end{align*}
so, by~\eqref{eq:Shen_A} we have 
\begin{align*}
    d(Q)^{-1} \leq C \max(m(\mathbf{x}),m(\sigma_{\alpha}(\mathbf{x}))) \leq C'm(\sigma_{\alpha}(\mathbf{x})) \text { for all }\mathbf{x} \in E. 
\end{align*}
Consequently, by Lemma~\ref{lem:difference_by_derivative}, we get
\begin{align*}
    I_1 &\leq \int_{E} \left|f(\sigma_{\alpha}(\mathbf{x}))\frac{\phi_{Q}(\mathbf{x})-\phi_{Q}(\sigma_{\alpha}(\mathbf{x}))}{\langle \mathbf{x},\alpha\rangle}\right|^2\,dw(\mathbf{x}) \leq C\int_{E}|f(\sigma_{\alpha}(\mathbf{x}))|^2d(Q)^{-2}\,dw(\mathbf{x}) \\&\leq C' \int_{\mathcal{O}(Q^{*})}|f(\sigma_{\alpha}(\mathbf{x}))|^2m(\sigma_{\alpha}(\mathbf{x}))^2\,dw(\mathbf{x})=C'\int_{\mathcal{O}(Q^{*})}|f(\mathbf{x})|^{2}m(\mathbf{x})^2\,dw(\mathbf{x}).
\end{align*}
In order to estimate $I_2$, thanks to property $0 \leq |\phi_{Q}(\mathbf{x})-\phi_{Q}(\sigma_{\alpha}(\mathbf{x}))| \leq 2$, we write
\begin{align*}
    I_2 \leq 4\int_{\mathcal{O}(Q^{*}) \cap \{\mathbf{x}\;:\;\sqrt{2}|\langle \mathbf{x},\alpha \rangle| > m(\mathbf{x})^{-1}\}}|f(\sigma_{\alpha}(\mathbf{x}))|^2|\langle \mathbf{x}, \alpha \rangle|^{-2}\,dw(\mathbf{x}).
\end{align*}
Note that, thanks to~\eqref{eq:Shen_C}, for $\mathbf{x} \in \mathcal{O}(Q^{*})$ such that $\sqrt{2}|\langle \mathbf{x},\alpha \rangle |=\|\mathbf{x}-\sigma_{\alpha}(\mathbf{x})\|>m(\mathbf{x})^{-1}$ we have
\begin{align*}
    m(\mathbf{x}) \leq Cm(\sigma_{\alpha}(\mathbf{x}))\left(1+m(\mathbf{x})\|\mathbf{x}-\sigma_{\alpha}(\mathbf{x})\|\right)^{\frac{\kappa}{1+\kappa}} \leq C'm(\sigma_{\alpha}(\mathbf{x}))m(\mathbf{x})|\langle \mathbf{x}, \alpha \rangle|,
\end{align*}
which lead us to
\begin{align*}
    I_2 \leq C\int_{\mathcal{O}(Q^{*})}|f(\sigma_{\alpha}(\mathbf{x}))|^2 m(\sigma_{\alpha}(\mathbf{x}))^2|\langle \mathbf{x},\alpha\rangle|^{2}|\langle \mathbf{x},\alpha\rangle|^{-2}\,dw(\mathbf{x})=C\int_{\mathcal{O}(Q^{*})}|f(\mathbf{x})|^{2}m(\mathbf{x})^2\,dw(\mathbf{x}),
\end{align*}
which ends the proof.
\end{proof}

\begin{proof}[Proof of Theorem~\ref{theo:Fefferman-Phong}]
Suppose first that
\begin{equation}\label{eq:a_priori}
    \int_{\mathbb{R}^N}|f(\mathbf{x})|^2m(\mathbf{x})^2\,dw(\mathbf{x})<\infty.
\end{equation}
Let $\psi \in C^{\infty}_c(\mathbb{R}^N)$ be a radial non-negative function such
that $\int_{\mathbb{R}^N}\psi\,dw=1$ and $\supp \psi \subseteq B(0,1)$, and let $A>1$ be a large constant (it will be chosen later). For $Q \in \mathcal{Q}$ we define the following scaled version of $\psi$:
\begin{align*}
    \psi_{Q}^{A}(\mathbf{x})=(A^{-1}d(Q))^{-\mathbf{N}}\psi(Ad(Q)^{-1}\mathbf{x}).
\end{align*}
It follows by Corollary~\ref{coro:Roesler} that
\begin{align*}
    |\mathcal{F}\psi(\xi)-1| \leq C\|\xi\|,
\end{align*}
consequently, by Plancherel's theorem (see~\eqref{eq:Plancherel}) and Lemma~\ref{lem:gradient},
\begin{equation}\label{eq:diff_final}
\begin{split}
    &\int_{Q^{*}}|\psi_Q^{A}*(\phi_Qf)(\mathbf{x})-(\phi_{Q}f)(\mathbf{x})|^2\,dw(\mathbf{x}) \leq CA^{-2}d(Q)^{2}\sum_{j=1}^{N}\int_{\mathcal{O}(Q^{*})}|T_j(\phi_{Q}f)(\mathbf{x})|^2 \,dw(\mathbf{x}) \\&\leq C'A^{-2}d(Q)^2\left(\sum_{j=1}^{N}\int_{Q^{*}}|T_jf(\mathbf{x})|^2\,dw(\mathbf{x})+\int_{\mathcal{O}(Q^{*})}|f(\mathbf{x})|^2m(\mathbf{x})^2\,dw(\mathbf{x})\right).
\end{split}
\end{equation}
The first inequality in \eqref{eq:diff_final} can be thought as a counterpart of the Poincar\'e inequality (cf. \eqref{eq:Poincare}). Furthermore, by Lemma~\ref{lem:nonradial_estimation} and the fact that by the doubling property of $w$ we have $w(B(\mathbf{x},d(Q))) \sim w(Q)$ for all $\mathbf{x} \in Q^{*}$, we obtain
\begin{equation}\label{eq:psi_phi_conv}
\begin{split}
    &\int_{Q^{*}}|\psi_{Q}^{A}*(\phi_{Q}f)(\mathbf{x})|^2\,dw(\mathbf{x})=\int_{Q^{*}}\left|\int_{Q^{*}}\tau_{\mathbf{x}}\psi_Q^{A}(-\mathbf{y})(\phi_{Q}f)(\mathbf{y})\,dw(\mathbf{y})\right|^2\,dw(\mathbf{x}) \\&\leq C \int_{Q^{*}}\frac{w(B(\mathbf{x},d(Q)))^2}{w(B(\mathbf{x},A^{-1}d(Q)))^2}\frac{1}{w(B(\mathbf{x},d(Q)))^2}\,dw(\mathbf{x})\|\phi_{Q}f\|^2_{L^1(dw)}\leq CA^{2\mathbf{N}}\frac{1}{w(Q^{*})}\|\phi_Q f\|^2_{L^1(dw)}.
\end{split}
\end{equation}
Let $\varepsilon>0$ (it will be chosen later) and let $E_{\varepsilon}$ be defined as in~\eqref{eq:E_def}. We write
\begin{equation}\label{eq:E_split}
    \frac{A^{2\mathbf{N}}}{w(Q^{*})}\|\phi_Q f\|^2_{L^1(dw)}=\frac{A^{2\mathbf{N}}}{w(Q^{*})}\|\phi_Q f\|^2_{L^1(E_{\varepsilon},dw)}+\frac{A^{2\mathbf{N}}}{w(Q^{*})}\|\phi_Q f\|^2_{L^1(Q^{*}\setminus E_{\varepsilon},dw)}.
\end{equation}
By the Cauchy--Schwarz inequality and Lemma~\ref{lem:E_set} we have
\begin{equation}\label{eq:E_1}
    A^{2\mathbf{N}}\frac{1}{w(Q^{*})}\|\phi_Q f\|^2_{L^1(E_{\varepsilon},dw)} \leq CA^{2\mathbf{N}}\varepsilon^{\eta}\|\phi_{Q}f\|_{L^2(dw)}^2.
\end{equation}
Next, by the definition of $E_{\varepsilon}$ (see~\eqref{eq:E_def}) and Cauchy--Schwarz inequality we get
\begin{equation}\label{eq:E_2}
    A^{2\mathbf{N}}\frac{1}{w(Q^{*})}\|\phi_Q f\|^2_{L^1(Q^{*}\setminus E_{\varepsilon},dw)} \leq CA^{2\mathbf{N}}d(Q)^{2}\varepsilon^{-1}\int_{Q^{*}}V(\mathbf{x})|(\phi_{Q}f)(\mathbf{x})|^2\,dw(\mathbf{x}).
\end{equation}
Combining~\eqref{eq:psi_phi_conv},~\eqref{eq:E_split},~\eqref{eq:E_1}, and~\eqref{eq:E_2} we get
\begin{equation}\label{eq:psi_phi_final}
    \int_{Q^{*}}|\psi_{Q}^{A}*(\phi_{Q}f)(\mathbf{x})|^2\,dw(\mathbf{x}) \leq CA^{2\mathbf{N}}\left(\varepsilon^{\eta}\|\phi_{Q}f\|_{L^2(dw)}^2+\frac{d(Q)^2}{\varepsilon}\int_{Q^{*}}V(\mathbf{x})|(\phi_{Q}f)(\mathbf{x})|^2\,dw(\mathbf{x})\right).
\end{equation}
Consequently, by~\eqref{eq:diff_final} and~\eqref{eq:psi_phi_final} we get
\begin{align*}
    \|\phi_{Q}f\|_{L^2(dw)}^2 &\leq CA^{-2}d(Q)^2\left(\sum_{j=1}^{N}\int_{Q^{*}}|T_jf(\mathbf{x})|^2\,dw(\mathbf{x})+\int_{\mathcal{O}(Q^{*})}|f(\mathbf{x})|^2m(\mathbf{x})^2\,dw(\mathbf{x})\right)\\&+CA^{2\mathbf{N}}\left(\varepsilon^{\eta}\|\phi_{Q}f\|_{L^2(dw)}^2+d(Q)^2\varepsilon^{-1}\int_{Q^{*}}V(\mathbf{x})|(\phi_{Q}f)(\mathbf{x})|^2\,dw(\mathbf{x})\right),
\end{align*}
which for $\varepsilon=\left(\frac{1}{2}C^{-1}A^{-2\mathbf{N}}\right)^{1/\eta}$ lead us to
\begin{equation}
    \begin{split}
         \|\phi_{Q}f\|_{L^2(dw)}^2 &\leq C'A^{-2}d(Q)^2\left(\sum_{j=1}^{N}\int_{Q^{*}}|T_jf(\mathbf{x})|^2\,dw(\mathbf{x})+\int_{\mathcal{O}(Q^{*})}|f(\mathbf{x})|^2m(\mathbf{x})^2\,dw(\mathbf{x})\right)\\&+C'A^{2\mathbf{N}}d(Q)^{2}\varepsilon^{-1}\int_{Q^{*}}V(\mathbf{x})|(\phi_{Q}f)(\mathbf{x})|^2\,dw(\mathbf{x}).
    \end{split}
\end{equation}
If we divide both sides by $d(Q)^{2}$ and then use Fact~\ref{fact:m_d_compare}, we get
\begin{equation}\label{eq:before_summing}
    \begin{split}
         &\int_{Q^{*}}|(\phi_{Q}f)(\mathbf{x})|^2m(\mathbf{x})^2\,dw(\mathbf{x}) \leq CA^{-2}\sum_{j=1}^{N}\int_{Q^{*}}|T_jf(\mathbf{x})|^2\,dw(\mathbf{x})\\&+CA^{-2}\int_{\mathcal{O}(Q^{*})}|f(\mathbf{x})|^2m(\mathbf{x})^2\,dw(\mathbf{x})+CA^{2\mathbf{N}}\varepsilon^{-1}\int_{Q^{*}}V(\mathbf{x})|(\phi_{Q}f)(\mathbf{x})|^2\,dw(\mathbf{x}).
    \end{split}
\end{equation}
Summing up over all $Q \in \mathcal{Q}$ we get
\begin{align*}
     \int_{\mathbb{R}^N}|f(\mathbf{x})|^2m(\mathbf{x})^2\,dw(\mathbf{x}) &\leq CA^{-2}\left(\sum_{j=1}^{N}\int_{\mathbb{R}^N}|T_jf(\mathbf{x})|^2\,dw(\mathbf{x})+|G|\int_{\mathbb{R}^N}|f(\mathbf{x})|^2m(\mathbf{x})^2\,dw(\mathbf{x})\right)\\&+CA^{2\mathbf{N}}\varepsilon^{-1}\int_{\mathbb{R}^N}V(\mathbf{x})|f(\mathbf{x})|^2\,dw(\mathbf{x}).
\end{align*}
Taking into account~\eqref{eq:a_priori} and taking $A$ large enough we obtain the claim for $f$ satisfying~\eqref{eq:a_priori}. For general case, we take a radial function $\eta \in C^{\infty}_c(\mathbb{R}^N)$ such that $0 \leq \eta \leq 1$, $\eta(\mathbf{x})=1$ for all $\|\mathbf{x}\| \leq 1$, $\eta(\mathbf{x})=0$ for all $\|\mathbf{x}\|>2$, and 
\begin{align*}
    |\partial_{j}\eta(\mathbf{x})| \leq 2 \text{ for all }\mathbf{x} \in \mathbb{R}^N \text{ and }j \in \{1,2,\ldots,N\}.
\end{align*}
For $f \in \mathcal{D}(\mathbf{Q})$ and $n \in \mathbb{N}$ we define $f_n(\mathbf{x})=f(\mathbf{x})\eta(\mathbf{x}/n)$. Note that by Lemma~\ref{eq:general_Leibniz} we have $f_n \in \mathcal{D}(\mathbf{Q})$. Moreover, thanks to the fact that $f \in L^2(dw)$ and~\eqref{eq:Shen_B}, the condition~\eqref{eq:a_priori} is satisfied for $f_n$. Therefore, by~\eqref{eq:Fefferman-Phong} for $f_n$, we get
\begin{equation}\label{eq:FP_for_approx}
    \int_{\mathbb{R}^N}|f(\mathbf{x})|^2m(\mathbf{x})^2\,dw(\mathbf{x})=\lim_{n \to \infty} \int_{\mathbb{R}^N}|f_n(\mathbf{x})|^2m(\mathbf{x})^2\,dw(\mathbf{x}) \leq C\lim_{n \to \infty}\mathbf{Q}(f_n,f_n).
\end{equation}
Clearly, 
\begin{equation}\label{eq:form_closed_assumption_1}
\lim_{n \to \infty}\|f-f_n\|_{L^2(dw)}=0.   
\end{equation}
Moreover, thanks to the definition of $\eta$, the fact that $f,T_jf \in L^2(dw)$, and Lemma~\ref{lem:general_Leibniz}, we have
\begin{equation}\label{eq:form_der_part}
\begin{split}
    &\lim_{n_1,n_2 \to \infty}\int_{\mathbb{R}^N}|T_j(f_{n_1}-f_{n_2})(\mathbf{x})|^2\,dw(\mathbf{x}) \leq 2\lim_{n_1, n_2 \to \infty}\int_{\mathbb{R}^N}|T_jf(\mathbf{x})|^2|\eta(\mathbf{x}/n_1)-\eta(\mathbf{x}/n_2)|^2\,dw(\mathbf{x})\\&+4\lim_{n_1,n_2 \to \infty}\int_{\mathbb{R}^N}|f(x)|^2(|\partial_j (\eta(\mathbf{x}/n_1))|^2+|\partial_j (\eta(\mathbf{x}/n_2))|^2)\,dw(\mathbf{x})\\&\leq 2\lim_{n_1,n_2 \to \infty}\int_{\min(n_1,n_2) \leq \|\mathbf{x}\| \leq 2\max(n_1,n_2)}|T_jf(\mathbf{x})|^2\,dw(\mathbf{x})\\&+32\lim_{n_1,n_2 \to \infty}\int_{\mathbb{R}^N}|f(\mathbf{x})|^2(n_1^{-2}+n_2^{-2})\,dw(\mathbf{x})=0.
\end{split}
\end{equation}
Similarly, $V(\mathbf{x})^{1/2}f(\mathbf{x}) \in L^2(dw(\mathbf{x}))$, so
\begin{equation}\label{eq:form_potential_part}
\begin{split}
    &\lim_{n_1,n_2 \to \infty}\int_{\mathbb{R}^N}V(\mathbf{x})|(f_{n_1}-f_{n_2})(\mathbf{x})|^2\,dw(\mathbf{x}) \\&= \lim_{n_1, n_2 \to \infty}\int_{\mathbb{R}^N}V(\mathbf{x})|f(\mathbf{x})|^2|\eta(\mathbf{x}/n_1)-\eta(\mathbf{x}/n_2)|^2\,dw(\mathbf{x})\\&\leq \lim_{n_1,n_2 \to \infty}\int_{\min(n_1,n_2) \leq \|\mathbf{x}\| \leq 2\max(n_1,n_2)}V(\mathbf{x})|f(\mathbf{x})|^2\,dw(\mathbf{x})=0.
\end{split}
\end{equation}
Consequently, by~\eqref{eq:form_der_part} and~\eqref{eq:form_potential_part} we have
\begin{equation}\label{eq:form_closed_assumption_2}
    \lim_{n_1,n_2 \to \infty}\mathbf{Q}(f_{n_1}-f_{n_2},f_{n_1}-f_{n_2})=0.
\end{equation}
By~\cite[Lemma 4.1]{AH} the form $\mathbf{Q}$ is closed, so by~\eqref{eq:form_closed_assumption_1} and~\eqref{eq:form_closed_assumption_2} we get
\begin{align*}
    \lim_{n \to \infty}\mathbf{Q}(f_n,f_n)=\mathbf{Q}(f,f),
\end{align*}
which, thanks to~\eqref{eq:FP_for_approx}, ends the proof.
\end{proof}

\part{Hardy spaces associated with Dunkl--Schr\"odinger operator.}\label{part:Hardy}

\section{Statement of the results} \label{sec:Hardy_pre}

\subsection{Background to the subject}

The classical real Hardy spaces $H^p$ in $\mathbb R^N$ occurred as boundary values of harmonic functions on $\mathbb{R}_{+}\times \mathbb{R}^{N}$  satisfying generalized Cauchy--Riemann equations together with certain $L^p$ bound conditions (see e.g. Stein--Weiss \cite{SW}). In the seminal paper of Fefferman and Stein~\cite{FS} the spaces $H^p$ were characterized by means of real analysis. One of the possible characterization assets that  a tempered distribution $f$ belongs to the $H^p(\mathbb R^N)$, $0<p<\infty$, if and only if the maximal function
$\sup_{t>0} |\mathbf{h}_t \ast f(\mathbf{x})|$ belongs to $L^p(\mathbb R^N)$, where $\mathbf{h}_t$ is the heat kernel of the semigroup $e^{t\Delta_{\text{eucl}}}$. An important contribution to the theory is the atomic decomposition proved by Coifman \cite{Coifman} for $N=1$  and Latter \cite{Latter} in higher dimensions, which says that every element of $H^p$ can be written as an (infinite) combination of special simple functions called atoms. These characterizations led to generalizations of the Hardy spaces on  spaces of homogeneous type, in particular, to $H^p$ spaces associated with semigroups of linear operators. In \cite{ADzH} (see also  \cite{Anker15},~\cite{DzH1}) a theory of Hardy spaces $H^1$ in the rational Dunkl setting parallel to the classical one was developed. The purpose of the remaining part of the paper is to study an $H^1_L$ space related to $L$. Our starting definition is that by means of the maximal function for the semigroup $e^{-tL}$. Then we shall prove that the space admits a special atomic decomposition. This result generalizes one of \cite{Hejna} where $H^1_L$ for  the Dunkl harmonic oscillator $-\Delta+\|\mathbf x\|^2$ was consider.   In~\cite{HMMLY} the authors provided a general approach to the theory of Hardy spaces associated with semigroups satisfying Davies--Gaffney estimates and in particular Gaussian bounds.  We want to  emphasize that the integral kernel for  the Dunkl--Laplace semigroup does not satisfy the  Gaussian bounds. Therefore the methods developed in~\cite{HMMLY}  cannot be directly applied.

\subsection{Hardy spaces associated with \texorpdfstring{$L$}{L}}
Let us introduce the notion of the Hardy space associated with the operator $L$.

\begin{definition}\normalfont
 Let $f \in L^{1}(dw)$. We say that $f$ belongs to the \textit{Hardy space} $H^1_{L}$ associated with operator $L $ if and only if
\begin{equation}
f^{*}(\mathbf{x})=\sup_{t>0}|K_tf(\mathbf{x})|
\end{equation}
belongs to $L^1(dw)$. The norm in the space is given by
\begin{equation}
\|f\|_{H^1_{L}}=\|f^{*}\|_{L^{1}(dw)}.
\end{equation}
\end{definition}

Let $\mathcal{Q}$ be a collection of closed cubes with parallel sides whose interiors are disjoint such that $\bigcup_{Q \in \mathcal{Q}}Q=\mathbb{R}^N$. Let us remind that $d(Q)$ denotes the side-length of cube $Q$ and we denote by $Q^{*}$ the cube with the same center as $Q$ such that $d(Q^{*})=2d(Q)$. Assume that this family satisfies the following finite overlapping condition: 
\begin{equation}\label{eq:finite_overlap}\tag{F}
    \left(\exists C_0>0\right)\left(\forall Q_1,Q_2 \in \mathcal{Q}\right)\, Q_1^{****} \cap Q_2^{****}\neq \emptyset \Rightarrow C_0^{-1}d(Q_1) \leq d(Q_2) \leq C_0 d(Q_1).
\end{equation}

We define the atomic Hardy space associated with the collection $\mathcal{Q}$ (see~\cite{DZ_Studia}).
\begin{definition}\label{def:atomic}\normalfont
A measurable function $a(\mathbf{x})$ is called an \textit{atom for the Hardy space} $H^{1,{{\rm{at}}}}_{\mathcal{Q}}$  associated with the collection of cubes $\mathcal{Q}$ if
\begin{enumerate}[(A)]
\item{$\supp a \subseteq B(\mathbf{x}_0,r) \subseteq Q^{****}$ for some $Q \in \mathcal{Q}$, $\mathbf{x}_0 \in \mathbb{R}^N$, and $r>0$,}\label{numitem:supports}
\item{$\sup_{\mathbf{y} \in \mathbb{R}^N}|a(\mathbf{y})|\leq w(B(\mathbf{x}_0,r))^{-1}$,}\label{numitem:size}
\item{if $r < d(Q)$, then $\int_{\mathbb{R}^N}a(\mathbf{x})\,dw(\mathbf{x})=0$.}\label{numitem:cancellation}
\end{enumerate}
 The \textit{atomic Hardy space} $H^{1,{\rm{at}}}_{\mathcal{Q}}$ associated with the collection $\mathcal{Q}$ is the space of functions $f \in L^1(dw)$ which admit a representation of the form
\begin{equation}
f(\mathbf{x})=\sum_{j=1}^{\infty}c_j a_j(\mathbf{x}),
\end{equation}
where $c_j \in \mathbb{C}$ and $a_j$ are atoms for the Hardy space $H^{1,{\rm{at}}}_{\mathcal{Q}}$ such that $\sum_{j=1}^{\infty}|c_j|<\infty$. The space $H^{1,{\rm{at}}}_{\mathcal{Q}}$ is a Banach space with the norm
\begin{equation}
\|f\|_{H^{1,{\rm{at}}}_{\mathcal{Q}}}=\inf\left\{\sum_{j=1}^{\infty}|c_j|:  f(\mathbf{x})=\sum_{j=1}^{\infty}c_j a_j(\mathbf{x}) \text{ and } a_j \text{ are }H^{1,{\rm{at}}}_{\mathcal{Q}}\text{ atoms}\right\}.
\end{equation}
\end{definition}

Inspired by~\cite{DZ_Studia}, we consider the following two additional conditions on $\mathcal{Q}$ and $V$:
\begin{equation}\label{eq:K}\tag{K}
\begin{split}
    (\exists C,\delta>0)(\forall \mathbf{x} \in \mathbb{R}^N, \, Q \in \mathcal{Q}, \ t \leq d(Q)^2) \int_0^{2t}\int_{Q^{***}}V(\mathbf{y})\mathcal{G}_{2s/c}(\mathbf{x},\mathbf{y})\,dw(\mathbf{y})\,ds \leq C\left(\frac{t}{d(Q)^2}\right)^{\delta},
\end{split}
\end{equation}
where $c>0$ is the constant from Theorem~\ref{theorem:heat},
\begin{equation}\label{eq:D}\tag{D}
    (\exists C,\varepsilon>0) (\forall Q \in \mathcal{Q}, \,s \in \mathbb{N})\sup_{\mathbf{y} \in Q^{****}}\int_{\mathbb{R}^N}k_{2^{s}d(Q)^2}(\mathbf{x},\mathbf{y})\,dw(\mathbf{x}) \leq Cs^{-1-\varepsilon}.
\end{equation}

The next theorem is one of the main result of the paper. We provide its proof in Section~\ref{sec:proof}.

\begin{theorem}\label{theorem:main}
Assume that the conditions~\eqref{eq:finite_overlap},~\eqref{eq:D}, and~\eqref{eq:K} hold for $V$ and $\mathcal{Q}$. There is a constant $C>0$ such that for all $f \in L^1(dw)$ we have
\begin{equation}
    C^{-1}\|f\|_{H^{1,{\rm{at}}}_{\mathcal{Q}}} \leq \|f\|_{H^1_{L}} \leq C\|f\|_{H^{1,{\rm{at}}}_{\mathcal{Q}}}.
\end{equation}
\end{theorem}

It can be checked that the conditions~\eqref{eq:finite_overlap},~\eqref{eq:D}, and~\eqref{eq:K} hold for potentials $V$ satisfying the reverse H\"older inequality with $q > \frac{\mathbf{N}}{2}$ and the associated collection of cubes~\eqref{eq:stopping_time}, so we obtain the following corollary.

\begin{corollary}\label{coro:main}
Assume that the potential $V$ satisfies the reverse H\"older inequality~\eqref{eq:reverse_Holder}. There is a constant $C>0$ such that for all $f \in L^1(dw)$ we have
\begin{align*}
    C^{-1}\|f\|_{H^{1,{\rm{at}}}_{\mathcal{Q}}} \leq \|f\|_{H^1_{L}} \leq C\|f\|_{H^{1,{\rm{at}}}_{\mathcal{Q}}},
\end{align*}
where $\mathcal{Q}$ is the collection of cubes defined in~\eqref{eq:stopping_time}.
\end{corollary}
Corollary~\ref{coro:main} is proved in Section~\ref{sec:verification}, where the conditions~\eqref{eq:finite_overlap},~\eqref{eq:D}, and~\eqref{eq:K} are verified.

\section{Local Hardy spaces}

The following two definitions are inspired by~\cite{Goldberg} (see also~\cite{Hejna}).
\begin{definition}\normalfont
Let $T>0$ and $f \in L^{1}(dw)$. We say that $f$ belongs to the \textit{local Hardy space} $H_{\rm{loc},T}^{1}$ associated with the Dunkl Laplacian if and only  if
\begin{equation}
f_{{\rm{loc}},T}^{*}(\mathbf{x})=\sup_{0<t \leq T^2}|{H_t}f(\mathbf{x})|
\end{equation}
belongs to $L^1(dw)$. The norm in the space is given by
\begin{equation}
\|f\|_{H^{1}_{{\rm{loc}},T}}=\|f^{*}_{{\rm{loc}},T}\|_{L^{1}(dw)}.
\end{equation}
\end{definition}

\begin{definition}\label{def:atomic_Dunkl}\normalfont
Let $T>0$. A function $a(\mathbf{x})$ is called an \textit{atom for the local Hardy space} $H^{1,{{\rm{at}}}}_{{\rm{loc}},T}$ if
\begin{enumerate}[(A)]
\item{$\supp a \subseteq B(\mathbf{x},r)$ for some $\mathbf{x} \in \mathbb{R}^N$ and $r>0$,}\label{numitem:support}
\item{$\sup_{\mathbf{y} \in \mathbb{R}^N}|a(\mathbf{y})| \leq w(B(\mathbf{x},r))^{-1}$,}\label{numitem:L_infty}
\item{If $r < T$, then $\int_{\mathbb{R}^N}a(\mathbf{x})\,dw(\mathbf{x})=0$.}\label{numitem:cancellations}
\end{enumerate}
A function $f$ belongs to the \textit{local Hardy space} $H_{{\rm{loc}},T}^{1,\rm{at}}$ if there are $c_j\in\mathbb{ C}$ and atoms $a_j$ for $H^{1,{{\rm{at}}}}_{{\rm{loc}},T}$ such that $\sum_{j=1}^{\infty}|c_j|<\infty$,
\begin{equation}
\label{eq:atomic_representation_local}
f=\sum_{j=1}^{\infty}c_j\,a_j\,.
\end{equation}
{In this case,} set
$
\|f\|_{H^{1,{{\rm{at}}}}_{{\rm{loc}},T}}=\inf\,\Bigl\{\,\sum_{j=1}^{\infty}|c_j|\,\Bigr\}\,,
$
where the infimum is taken over all representations {\eqref{eq:atomic_representation_local}}.
\end{definition}

The following proposition was proved in~\cite{Hejna} and its proof follows the pattern from~\cite{Goldberg}.

\begin{proposition}\label{propo:Goldberg}
The spaces $H^{1,{{\rm{at}}}}_{{\rm{loc}},T}$ and $H^1_{{\rm{loc}},T}$ coincide and their norms are equivalent. Moreover, there exists a constant $C>0$ such that for any $T>0$ if $f \in H^{1,{{\rm{at}}}}_{{\rm{loc}},T}$ and $\supp f \subseteq B(\mathbf{y}_0,T)$, then there are $H^{1,{{\rm{at}}}}_{{\rm{loc}},T}$ atoms $a_j$ such that $\supp a_j \subseteq B(\mathbf{y}_0,4T)$ and
\begin{equation}
f=\sum_{j=1}^{\infty}c_ja_j, \qquad \sum_{j=1}^{\infty}|c_j| \leq C\|f\|_{H^{1,{{\rm{at}}}}_{{\rm{loc}},T}}.
\end{equation}
\end{proposition}

\section{Auxiliary lemmas}
Lemmas in this section are inspired by~\cite{DZ_Studia}. It turns out that the presence of the factor "$\Big(1+\frac{\|\mathbf{x}-\mathbf{y}\|^2}{t}\Big)^{-1}$" in the estimate from Theorem~\ref{theorem:heat} is crucial in the proof of Theorem~\ref{theorem:main} and its proper usage is the main difficulty and difference between the proofs here and in~\cite{DZ_Studia}. Let $\{\phi_{Q}\}_{Q \in \mathcal{Q}}$ be the resolution of identity associated with the collection $\mathcal{Q}$, which satisfies the analogous properties to that from Section~\ref{sec:Fefferman--Phong} (see e.g.~\eqref{eq:partition_derivative}).
\begin{lemma}
There is a constant $C>0$ such that for all $Q \in \mathcal{Q}$  and $f \in L^1(dw)$ we have
\begin{equation}\label{eq:outside_H}
    \int_{\mathbb{R}^N \setminus Q^{**}}\sup_{0 <t \leq d(Q)^2}|H_t(\phi_{Q}f)(\mathbf{x})|\,dw(\mathbf{x}) \leq C \|\phi_{Q}f\|_{L^1(dw)},
\end{equation}
\begin{equation}\label{eq:outside_K}
    \int_{\mathbb{R}^N \setminus Q^{**}}\sup_{0 <t \leq d(Q)^2}|K_t(\phi_{Q}f)(\mathbf{x})|\,dw(\mathbf{x}) \leq C \|\phi_{Q}f\|_{L^1(dw)}.
\end{equation}
\end{lemma}

\begin{proof}
We will prove just~\eqref{eq:outside_H}, thanks to~\eqref{eq:kernels_compare} the proof of~\eqref{eq:outside_K} is the same. We have
\begin{equation}\label{eq:max_split_j}
\begin{split}
    &\int_{\mathbb{R}^N \setminus Q^{**}}\sup_{0 <t \leq d(Q)^2}|H_t(\phi_{Q}f)(\mathbf{x})|\,dw(\mathbf{x}) \leq \sum_{j=0}^{\infty}\int_{\mathbb{R}^N \setminus Q^{**}}\sup_{2^{-j-1}d(Q)^2 <t \leq 2^{-j}d(Q)^2}|H_t(\phi_{Q}f)(\mathbf{x})|\,dw(\mathbf{x})\\&\leq \sum_{j=0}^{\infty}\int_{\mathbb{R}^N \setminus Q^{**}}\sup_{2^{-j-1}d(Q)^2 <t \leq 2^{-j}d(Q)^2}\left(\int_{Q^{*}}h_t(\mathbf{x},\mathbf{y})|(\phi_{Q}f)(\mathbf{y})|\,dw(\mathbf{y})\right)\,dw(\mathbf{x}).
\end{split}
\end{equation}
Thanks to Theorem~\ref{theorem:heat} and the fact that for $\mathbf{x} \in \mathbb{R}^N \setminus Q^{**}$ and $\mathbf{y} \in Q^{*}$ we have $\|\mathbf{x}-\mathbf{y}\| \geq d(Q)$, so we obtain
\begin{align*}
    &\int_{\mathbb{R}^N \setminus Q^{**}}\sup_{2^{-j-1}d(Q)^2 <t \leq 2^{-j}d(Q)^2}\left(\int_{Q^{*}}h_t(\mathbf{x},\mathbf{y})|(\phi_{Q}f)(\mathbf{y})|\,dw(\mathbf{y})\right)\,dw(\mathbf{x}) \\&\leq C\frac{2^{-j}d(Q)^2}{d(Q)^2} \int_{Q^{*}}|(\phi_{Q}f)(\mathbf{y})|\int_{\mathbb{R}^N \setminus Q^{**}}\frac{1}{w(B(\mathbf{x},2^{-j/2}d(Q)))}e^{-cd(\mathbf{x},\mathbf{y})^2/(2^{-j}d(Q)^2)}\,dw(\mathbf{x})\,dw(\mathbf{y})\\&\leq C'2^{-j}\|\phi_{Q}f\|_{L^{1}(dw)}.
\end{align*}
The latest estimate together with~\eqref{eq:max_split_j} implies the claim.
\end{proof}

\begin{corollary}\label{coro:outside}
There is a constant $C>0$ such that for every $Q \in \mathcal{Q}$ and $f \in L^1(dw)$ we have
\begin{equation}
    \|\phi_{Q}f\|_{H_{{\rm{loc}},d(Q)}^{1}} \leq C\|\sup_{0<t \leq d(Q)^2}|H_t(\phi_{Q}f)|\|_{L^1(Q^{**},\,dw)}+C\|\phi_{Q}f\|_{L^1(dw)}.
\end{equation}
\end{corollary}

For $Q \in \mathcal{Q}$ we define
\begin{equation}\label{eq:Q_prim}
    \mathcal{Q}'(Q)=\{Q' \in \mathcal{Q}\; : \; Q^{***} \cap (Q')^{***} \neq \emptyset \},
\end{equation}
\begin{equation}\label{eq:Q_bis}
    \mathcal{Q}''(Q)=\{Q'' \in \mathcal{Q}\; : \; Q^{***} \cap (Q'')^{***} = \emptyset \}.
\end{equation}

\begin{lemma}\label{lem:commutator}
There is a constant $C>0$ such that for every $Q \in \mathcal{Q}$ and $f \in L^1(\mathbb{R}^N)$ we have
\begin{equation}
    \left\|\sup_{0<t \leq d(Q)^2}|K_t(\phi_{Q}g)-\phi_{Q}K_t(g)|\right\|_{L^1(Q^{**},dw)} \leq C\sum_{Q' \in \mathcal{Q}'(Q)}\|\phi_{Q'}f\|_{L^1(dw)},
\end{equation}
where $g=\sum_{Q' \in \mathcal{Q}'(Q)}\phi_{Q'}f$.
\end{lemma}

\begin{proof}
Thanks to~\eqref{eq:partition_derivative}, then Theorem~\ref{theorem:heat} together with~\eqref{eq:kernels_compare} and~\eqref{eq:growth} we get
\begin{align*}
    &\sup_{0<t \leq d(Q)^2}|K_t(\phi_{Q}g)(\mathbf{x})-\phi_{Q}(\mathbf{x})K_tg(\mathbf{x})|=\sup_{0<t \leq d(Q)^2}\left|\int_{\mathbb{R}^N}(\phi_{Q}(\mathbf{y})-\phi_{Q}(\mathbf{x}))k_{t}(\mathbf{x},\mathbf{y})g(\mathbf{y})\,dw(\mathbf{y})\right| 
    \\&\leq C\sup_{0<t \leq d(Q)^2}\int_{\mathbb{R}^N}\frac{\|\mathbf{x}-\mathbf{y}\|}{d(Q)}k_{t}(\mathbf{x},\mathbf{y})|g(\mathbf{y})|\,dw(\mathbf{y})\\&\leq C\sup_{0<t \leq d(Q)^2}\int_{\mathbb{R}^N}\frac{\|\mathbf{x}-\mathbf{y}\|}{d(Q)}\frac{\sqrt{t}}{\|\mathbf{x}-\mathbf{y}\|}\frac{1}{w(B(\mathbf{y},\sqrt{t}))}e^{-cd(\mathbf{x},\mathbf{y})^2/t}|g(\mathbf{y})|\,dw(\mathbf{y})\\&\leq C\sum_{j=0}^{\infty}\sup_{2^{-j-1}d(Q)^2<t \leq 2^{-j}d(Q)^2}\int_{\mathbb{R}^N}\frac{\sqrt{t}}{d(Q)}\frac{1}{w(B(\mathbf{y},\sqrt{t}))}e^{-cd(\mathbf{x},\mathbf{y})^2/t}|g(\mathbf{y})|\,dw(\mathbf{y}) \\&\leq C\sum_{j=0}^{\infty}\int_{\mathbb{R}^N}\frac{2^{-j/2}d(Q)}{d(Q)}\frac{1}{w(B(\mathbf{y},2^{-j/2}d(Q)))}e^{-cd(\mathbf{x},\mathbf{y})^2/(2^{-j}d(Q)^2)}|g(\mathbf{y})|\,dw(\mathbf{y}).
\end{align*}
Consequently, by the Fubini theorem,
\begin{align*}
    \|\sup_{0<t \leq d(Q)^2}|K_t(\phi_{Q}g)(\mathbf{x})-\phi_{Q}(\mathbf{x})K_tg(\mathbf{x})|\|_{L^1(dw(\mathbf{x}))} \leq C\sum_{j=0}^{\infty}2^{-j/2}\|g\|_{L^1(dw)} \leq C\|g\|_{L^1}.
\end{align*}

\end{proof}

\begin{lemma}
Assume that $\mathcal{Q}$ and $V$ satisfy condition~\eqref{eq:D}. Then there is a constant $C>0$ such that for all $f \in L^1(dw)$ we have
\begin{equation}\label{eq:outside_sum}
    \sum_{Q \in \mathcal{Q}}\left\|\chi_{Q^{***}}(\cdot)\sup_{t>0}\left|K_t\left(\sum_{Q'' \in \mathcal{Q}''(Q)}\phi_{Q''}f\right)\right|\right\|_{L^1(dw)} \leq C\|f\|_{L^1(dw)}.
\end{equation}
\end{lemma}

\begin{proof}
Let us denote the left-hand side of~\eqref{eq:outside_sum} by $S$. Then by property~\eqref{eq:finite_overlap} we get
\begin{align*}
    S &\leq \sum_{Q \in \mathcal{Q}}\sum_{Q'' \in \mathcal{Q}''(Q)}\|\chi_{Q^{***}}(\cdot)\sup_{t>0}(K_t|\phi_{Q''}f|)\|_{L^1(dw)} \\&\leq \sum_{Q'' \in \mathcal{Q}}\sum_{Q \in \mathcal{Q}''(Q'')}\|\chi_{Q^{***}}(\cdot)\sup_{t>0}(K_t|\phi_{Q''}f|)\|_{L^1(dw)} \leq C \sum_{Q'' \in \mathcal{Q}}\|\sup_{t>0}(K_t|\phi_{Q''}f|)\|_{L^1(((Q'')^{**})^{c},dw)} \\&\leq C \sum_{Q'' \in \mathcal{Q}}\|\sup_{0<t<d(Q'')^2}(K_t|\phi_{Q''}f|)\|_{L^1(((Q'')^{**})^{c},dw)}\\&+\sum_{j=0}^{\infty}\sum_{Q'' \in \mathcal{Q}}\|\sup_{2^{j}d(Q'')^{2} \leq t < 2^{j+1}d(Q'')^2}(K_t|\phi_{Q''}f|)\|_{L^1(((Q'')^{**})^{c},dw)}=:S_1+S_2.
\end{align*}
The estimate $S_1 \leq C\|f\|_{L^1(dw)}$ follows by~\eqref{eq:outside_K} and~\eqref{eq:finite_overlap}. Furthermore, by the semigroup property and Theorem~\ref{theorem:heat} together with~\eqref{eq:kernels_compare}, for $2^{j}d(Q)^2 \leq t <2^{j+1}d(Q)^2$ we have
\begin{align*}
    &\int_{\mathbb{R}^N}k_t(\mathbf{x},\mathbf{y})|(\phi_{Q''}f)(\mathbf{y})|\,dw(\mathbf{y})\\&=\int_{\mathbb{R}^N}\int_{\mathbb{R}^N}k_{t-2^{j-1}d(Q'')^2}(\mathbf{x},\mathbf{z})k_{2^{j-1}d(Q'')^2}(\mathbf{z},\mathbf{y})\,dw(\mathbf{z})|(\phi_{Q''}f)(\mathbf{y})|\,dw(\mathbf{y})\\&\leq C \int_{\mathbb{R}^N}\int_{\mathbb{R}^N}\frac{1}{w(B(\mathbf{z},2^{j/2}d(Q'')))}e^{\frac{-cd(\mathbf{x},\mathbf{z})^2}{2^{j+1}d(Q'')^2}}k_{2^{j-1}d(Q'')^2}(\mathbf{z},\mathbf{y})\,dw(\mathbf{z})|(\phi_{Q''}f)(\mathbf{y})|\,dw(\mathbf{y}).
\end{align*}
Therefore, integrating over the $\mathbf{x}$-variable we obtain
\begin{align*}
    &\int_{((Q'')^{**})^{c}}\sup_{2^{j}d(Q'')^2 \leq t <2^{j+1}d(Q'')^2}\int_{\mathbb{R}^N}k_t(\mathbf{x},\mathbf{y})|(\phi_{Q''}f)(\mathbf{y})|\,dw(\mathbf{y})\,dw(\mathbf{x})\\&\leq C \int_{\mathbb{R}^N}|(\phi_{Q''}f)(\mathbf{y})|\int_{\mathbb{R}^N}k_{2^{j-1}d(Q'')^2}(\mathbf{z},\mathbf{y})\,dw(\mathbf{z})\,dw(\mathbf{y}).
\end{align*}
Consequently, by assumption~\eqref{eq:D}, we get
\begin{align*}
    S_2 \leq C\sum_{j=0}^{\infty}\sum_{Q'' \in \mathcal{Q}}j^{-1-\varepsilon}\|\phi_{Q''}f\|_{L^1(dw)} \leq C\|f\|_{L^1(dw)}.
\end{align*}
\end{proof}

\begin{lemma}
For all $f \in L^1(dw)$ we have
\begin{equation}\label{eq:VK}
    \int_{\mathbb{R}^N}\int_0^{\infty}V(\mathbf{x})K_s|f|(\mathbf{x})\,ds\,dw(\mathbf{x}) \leq \|f\|_{L^1(dw)}.
\end{equation}
\end{lemma}

\begin{proof}
The lemma is well-known. We provide the proof for the sake of completeness. By perturbation formula we have
\begin{align*}
    H_{t}|f|(\mathbf{x})-K_{t}|f|(\mathbf{x})=\int_0^{t}H_{t-s}VK_s|f|(\mathbf{x})\,ds,
\end{align*}
so, by~\eqref{eq:kernels_compare}, we have
\begin{equation}\label{eq:KV_1}
    \int_0^{t}\int_{\mathbb{R}^N}h_{t-s}(\mathbf{x},\mathbf{y})VK_s|f|(\mathbf{y})\,dw(\mathbf{y})\,ds \leq \int_{\mathbb{R}^N}h_t(\mathbf{x},\mathbf{y})|f|(\mathbf{y})\,dw(\mathbf{y}).
\end{equation}
Integrating~\eqref{eq:KV_1} with respect to the $\mathbf{x}$-variable, using the Fubini theorem and the fact that for all $v>0$ we have $\int_{\mathbb{R}^N}h_{v}(\mathbf{x},\mathbf{y})\,dw(\mathbf{x})=1$ (see~\eqref{eq:h_integral_1}), we get
\begin{align*}
    \int_{0}^{t}\int_{\mathbb{R}^N}V(\mathbf{y})K_s|f|(\mathbf{y})\,dw(\mathbf{y})\,ds \leq \|f\|_{L^1(dw)}.
\end{align*}
Letting $t \to \infty$ we obtain the lemma.
\end{proof}

\begin{lemma}\label{lem:difference}
Assume that $\mathcal{Q}$ and $V$ satisfy~\eqref{eq:K}. There is a constant $C>0$ such that for all $Q \in \mathcal{Q}$ and $f \in L^1(dw)$ we have
\begin{equation}\label{eq:difference}
    \|\sup_{0<t \leq d(Q)^2}|(H_t-K_t)(\phi_{Q}f)|\|_{L^1(dw)} \leq C\|\phi_{Q}f\|_{L^1(dw)}.
\end{equation}
\end{lemma}

\begin{proof}
Thanks to~\eqref{eq:outside_H} and~\eqref{eq:outside_K} it is enough to estimate 
\begin{align*}
 \|\sup_{0<t \leq d(Q)^2}|(H_t-K_t)(\phi_{Q}f)|\|_{L^1(Q^{**},dw)}.    
\end{align*}
By perturbation formula we write
\begin{equation}\label{eq:perturbation_split}
\begin{split}
    H_t(\phi_{Q}f)(\mathbf{x})-K_t(\phi_{Q}f)(\mathbf{x})&=\int_0^{t}\int_{\mathbb{R}^N}h_{t-s}(\mathbf{x},\mathbf{y})V(\mathbf{y})K_s(\phi_{Q}f)(\mathbf{y})\,dw(\mathbf{y})\,ds \\&=\int_0^{t}\int_{\mathbb{R}^N}h_{t-s}(\mathbf{x},\mathbf{y})V_1(\mathbf{y})K_s(\phi_{Q}f)(\mathbf{y})\,dw(\mathbf{y})\,ds\\&+\int_0^{t}\int_{\mathbb{R}^N}h_{t-s}(\mathbf{x},\mathbf{y})V_2(\mathbf{y})K_s(\phi_{Q}f)(\mathbf{y})\,dw(\mathbf{y})\,ds,
\end{split}
\end{equation}
where $V_1+V_2=V$ and $V_1=V\chi_{Q^{***}}$. In order to estimate the term with $V_2$, we use Theorem~\ref{theorem:heat} and the fact that for $\mathbf{y} \in \mathbb{R}^N \setminus Q^{***}$ and $\mathbf{x} \in Q^{**}$ we have $\|\mathbf{x}-\mathbf{y}\| \geq d(Q)$, so, for $\mathbf{x} \in Q^{**}$ we get
\begin{align*}
    &\sup_{0<t \leq d(Q)^2}\left| \int_0^{t}\int_{\mathbb{R}^N}h_{t-s}(\mathbf{x},\mathbf{y})V_2(\mathbf{y})K_s(\phi_{Q}f)(\mathbf{y})\,dw(\mathbf{y})\,ds\right| \\&\leq \sum_{j=0}^{\infty}\sup_{2^{-j-1}d(Q)^2<t \leq 2^{-j}d(Q)^2} \int_0^{t}\int_{\mathbb{R}^N}h_{t-s}(\mathbf{x},\mathbf{y})V_2(\mathbf{y})K_s(|\phi_{Q}f|)(\mathbf{y})\,dw(\mathbf{y})\,ds\\&=\sum_{j=0}^{\infty}\sup_{2^{-j-1}d(Q)^2<t \leq 2^{-j}d(Q)^2} \sum_{\ell=0}^{\infty}\int_{t-2^{-\ell}t}^{t-2^{-\ell-1}t}\int_{\mathbb{R}^N}h_{t-s}(\mathbf{x},\mathbf{y})V_2(\mathbf{y})K_s(|\phi_{Q}f|)(\mathbf{y})\,dw(\mathbf{y})\,ds \\&\leq C\sum_{j,\ell=0}^{\infty}\sup_{2^{-j-1}d(Q)^2 <t\leq 2^{-j}d(Q)^2}\int_{t-2^{-\ell}t}^{t-2^{-\ell-1}t}\int_{\mathbb{R}^N}\frac{t-s}{\|\mathbf{x}-\mathbf{y}\|^2}\frac{1}{w(B(\mathbf{y},\sqrt{t-s}))}\\&\times e^{-cd(\mathbf{x},\mathbf{y})^2/(t-s)}V_2(\mathbf{y})K_{s}|(\phi_{Q}f)|(\mathbf{y})\,dw(\mathbf{y})\,ds \\&\leq C\sum_{j,\ell=0}^{\infty}\int\limits_0^{\infty}\int\limits_{\mathbb{R}^N}\frac{2^{-j-\ell}d(Q)^2}{d(Q)^2}\frac{1}{w(B(\mathbf{y},2^{-(j+\ell)/2}d(Q)))}e^{-\frac{cd(\mathbf{x},\mathbf{y})^2}{(2^{-j-\ell}d(Q)^2)}}V(\mathbf{y})K_s|(\phi_{Q}f)|(\mathbf{y})\,dw(\mathbf{y})\,ds.
\end{align*}
Therefore, by the Fubini theorem and~\eqref{eq:VK} we obtain
\begin{align*}
    &\left\|\sup_{0<t \leq d(Q)^2} \int_0^{t}\int_{\mathbb{R}^N}h_{t-s}(\mathbf{x},\mathbf{y})V_2(\mathbf{y})K_s(\phi_{Q}f)(\mathbf{y})\,dw(\mathbf{y})\,ds \right\|_{L^1(Q^{**},\,dw)} \\ &\leq C\sum_{j,l=0}^{\infty}2^{-j-l}\int_0^{\infty}\int_{\mathbb{R}^N}V(\mathbf{y})K_s|(\phi_{Q}f)|(\mathbf{y})\,dw(\mathbf{y})\,ds \leq C\|\phi_{Q}f\|_{L^1(dw)}.
\end{align*}
In order to estimate the term containing $V_1$ in~\eqref{eq:perturbation_split}, we write
\begin{align*}
    \int_0^{t}\int_{\mathbb{R}^N}h_{t-s}(\mathbf{x},\mathbf{y})V_1(\mathbf{y})K_s(\phi_{Q}f)(\mathbf{y})\,dw(\mathbf{y})\,ds=\int_0^{t/2}\ldots+\int_{t/2}^{t}\ldots=:I_t(\mathbf{x})+J_t(\mathbf{x}).
\end{align*}
Clearly, by Theorem~\ref{theorem:heat} and the Fubini theorem, we get
\begin{align*}
    &\|\sup_{0<t\leq d(Q)^2}|I_t(\mathbf{x})|\|_{L^1(dw(\mathbf{x}))} \leq \sum_{j=0}^{\infty} \|\sup_{2^{-j-1}d(Q)^2<t\leq 2^{-j}d(Q)^2}|I_t(\mathbf{x})|\|_{L^1(dw(\mathbf{x}))} \\&\leq C\sum_{j=0}^{\infty}\int\limits_{Q^{**}}\int\limits_0^{2^{-j}d(Q)^2}\int\limits_{\mathbb{R}^N}\frac{1}{w(B(\mathbf{y},2^{-j/2}d(Q)))}e^{\frac{-cd(\mathbf{x},\mathbf{y})^2}{2^{-j}d(Q)^2}}V_1(\mathbf{y})K_s|(\phi_{Q}f)|(\mathbf{y})\,dw(\mathbf{y})\,ds\,dw(\mathbf{x}) \\&\leq C\sum_{j=0}^{\infty}\int_{0}^{2^{-j}d(Q)^2}\int_{\mathbb{R}^N}V_1(\mathbf{y})K_s|(\phi_{Q}f)|(\mathbf{y})\,dw(\mathbf{y})\,ds\\&=C\sum_{j=0}^{\infty}\int_0^{2^{-j}d(Q)^2}\int_{Q^{***}}V(\mathbf{y})\int_{\mathbb{R}^N}k_{s}(\mathbf{y},\mathbf{z})|(\phi_{Q}f)(\mathbf{z})|\,dw(\mathbf{z})\,dw(\mathbf{y})\,ds\\&\leq C'\sum_{j=0}^{\infty}\int_{\mathbb{R}^N}|(\phi_{Q}f)(\mathbf{z})|\left(\int_0^{2^{-j}d(Q)^2}\int_{Q^{***}}V(\mathbf{y})\mathcal{G}_{s/c}(\mathbf{y},\mathbf{z})\,dw(\mathbf{y})\,ds\right)\,dw(\mathbf{z}),
\end{align*}
where in the last step we have used ~\eqref{eq:kernels_compare} and Theorem~\ref{theorem:heat}. Consequently, by assumption~\eqref{eq:K}, we get
\begin{align*}
    &\|\sup_{0<t\leq d(Q)^2}|I_t(\mathbf{x})|\|_{L^1(dw(\mathbf{x}))} \leq C\sum_{j=0}^{\infty}2^{-j\delta}\|(\phi_{Q}f)\|_{L^1(dw)} \leq C\|(\phi_{Q}f)\|_{L^1(dw)}.
\end{align*}
Similarly, we write
\begin{equation}\label{eq:J_split}
     \|\sup_{0<t\leq d(Q)^2}|J_t(\mathbf{x})|\|_{L^1(dw(\mathbf{x}))} \leq \sum_{j=0}^{\infty} \|\sup_{2^{-j-1}d(Q)^2<t\leq 2^{-j}d(Q)^2}|J_t(\mathbf{x})|\|_{L^1(dw(\mathbf{x}))},
\end{equation}
then by changing of variables we have
\begin{align*}
    |J_t(\mathbf{x})|&\leq \int_0^{t/2}\int_{\mathbb{R}^N}h_{s}(\mathbf{x},\mathbf{y})V_1(\mathbf{y})K_{t-s}(|\phi_{Q}f|)(\mathbf{y})\,dw(\mathbf{y})\,ds\\&=\int_{\mathbb{R}^N}\int_0^{t/2}\int_{\mathbb{R}^N}h_{s}(\mathbf{x},\mathbf{y})V_1(\mathbf{y})k_{t-s}(\mathbf{y},\mathbf{z})(|\phi_{Q}f|)(\mathbf{z})\,dw(\mathbf{y})\,ds\,dw(\mathbf{z}),
\end{align*}
so, by Theorem~\ref{theorem:heat} and~\eqref{eq:kernels_compare} we get
\begin{equation}\label{eq:triple}
\begin{split}
    &\sup_{2^{-j-1}d(Q)^2<t \leq 2^{-j}d(Q)^2}|J_t(\mathbf{x})|\\& \leq
    C\int_{\mathbb{R}^N}\int_0^{t/2}\int_{\mathbb{R}^N}\mathcal{G}_{s/c}(\mathbf{x},\mathbf{y})V_1(\mathbf{y})\mathcal{G}_{2^{-j}d(Q)^2/c}(\mathbf{y},\mathbf{z})(|\phi_{Q}f|)(\mathbf{z})\,dw(\mathbf{y})\,ds\,dw(\mathbf{z}).
\end{split}
\end{equation}
Moreover, for $s \leq \frac{t}{2} \leq 2^{-j-1}d(Q)^2$ we have
\begin{align*}
    e^{-cd(\mathbf{x},\mathbf{y})^2/s}  e^{-cd(\mathbf{y},\mathbf{z})^2/(2^{-j}d(Q)^2)} &\leq  e^{-cd(\mathbf{x},\mathbf{y})^2/(2s)} e^{-cd(\mathbf{x},\mathbf{y})^2/(2^{-j}d(Q)^2)}  e^{-cd(\mathbf{y},\mathbf{z})^2/(2^{-j}d(Q)^2)}\\&\leq e^{-cd(\mathbf{x},\mathbf{y})^2/(2s)} e^{-cd(\mathbf{x},\mathbf{z})^2/(2^{-j+1}d(Q)^2)},
\end{align*}
so~\eqref{eq:triple} and the doubling property of $w$ lead us to
\begin{align*}
    &\sup_{2^{-j-1}d(Q)^2<t \leq 2^{-j}d(Q)^2}|J_t(\mathbf{x})| \\& \leq
    \int_{\mathbb{R}^N}\int_0^{t/2}\int_{\mathbb{R}^N}\mathcal{G}_{2s/c}(\mathbf{x},\mathbf{y})V_1(\mathbf{y})\frac{1}{w(B(\mathbf{z},2^{-j/2}d(Q)))}e^{-\frac{cd(\mathbf{x},\mathbf{z})^2}{2^{-j+1}d(Q)^2}}(|\phi_{Q}f|)(\mathbf{z})\,dw(\mathbf{y})\,ds\,dw(\mathbf{z}). 
\end{align*}
Furthermore, by assumption~\eqref{eq:K}, we get
\begin{equation}\label{eq:app_K}
    \sup_{2^{-j-1}d(Q)^2<t \leq 2^{-j}d(Q)^2}|J_t(\mathbf{x})|   \leq C2^{-j\delta}\int_{\mathbb{R}^N}\frac{1}{w(B(\mathbf{z},2^{-j/2}d(Q)))}e^{-\frac{cd(\mathbf{x},\mathbf{z})^2}{2^{-j+1}d(Q)^2}}(|\phi_{Q}f|)(\mathbf{z})\,dw(\mathbf{z}).
\end{equation}
Finally, integrating~\eqref{eq:app_K} with respect to $\mathbf{x}$-variable and taking~\eqref{eq:J_split} into account we are done.
\end{proof}

\section{Proof of Theorem \texorpdfstring{~\ref{theorem:main}}{ main }}\label{sec:proof}

\subsection{Proof of the inequality \texorpdfstring{$C^{-1}\|f\|_{H^{1,{\rm{at}}}_{\mathcal{Q}}} \leq \|f\|_{H^1_{L}}$}{first}}
Thanks to Corollary~\ref{coro:outside} and Lemma~\ref{lem:difference} we have
\begin{align*}
    &\sum_{Q \in \mathcal{Q}}\|\phi_{Q}f\|_{H^{1}_{{\rm{loc}}, d(Q)}} \leq C\sum_{Q \in \mathcal{Q}}\|\sup_{0<t \leq d(Q)^2}|H_t(\phi_{Q}f)|\|_{L^1(Q^{**},\,dw)}+C\|f\|_{L^1(dw)} \\&\leq C\sum_{Q \in \mathcal{Q}}\|\sup_{0<t \leq d(Q)^2}|(H_t-K_t)(\phi_{Q}f)|\|_{L^1(Q^{**},\,dw)}+C\sum_{Q \in \mathcal{Q}}\|\sup_{0<t \leq d(Q)^2}|K_t(\phi_{Q}f)|\|_{L^1(Q^{**},\,dw)}\\&+C\|f\|_{L^1(dw)}\leq C\|f\|_{L^1(dw)}+C\sum_{Q \in \mathcal{Q}}\|\sup_{0<t \leq d(Q)^2}|K_t(\phi_{Q}f)|\|_{L^1(Q^{**},\,dw)}.
\end{align*}
Then, by Lemma~\ref{lem:commutator} and~\eqref{eq:outside_sum} we get
\begin{align*}
    &\sum_{Q \in \mathcal{Q}}\|\sup_{0<t \leq d(Q)^2}|K_t(\phi_{Q}f)|\|_{L^1(Q^{**},\,dw)}\leq \sum_{Q \in \mathcal{Q}}\int_{Q^{**}}\sup_{0<t \leq d(Q)^2}|K_t(\phi_{Q}f)(\mathbf{x})|\,dw(\mathbf{x})\\&=\sum_{Q \in \mathcal{Q}}\int_{Q^{**}}\sup_{0<t \leq d(Q)^2}|K_t\big(\phi_{Q}\sum_{Q' \in \mathcal{Q}'(Q)}(\phi_{Q'}f)\big)(\mathbf{x})|\,dw(\mathbf{x})\\&\leq \sum_{Q \in \mathcal{Q}}\int_{Q^{**}}\sup_{0<t \leq d(Q)^2}|K_t\big(\phi_{Q}\sum_{Q' \in \mathcal{Q}'(Q)}(\phi_{Q'}f)\big)(\mathbf{x})-\phi_{Q}(\mathbf{x})\big(K_t(\sum_{Q' \in \mathcal{Q}'(Q)}(\phi_{Q'}f)\big)(\mathbf{x})|\,dw(\mathbf{x})\\&+ \sum_{Q \in \mathcal{Q}}\int_{Q^{**}}\sup_{0<t \leq d(Q)^2}|\phi_{Q}(\mathbf{x})\big(K_t(\sum_{Q'' \in \mathcal{Q}''(Q)}(\phi_{Q''}f)\big)(\mathbf{x})|\,dw(\mathbf{x})\\&+\sum_{Q \in \mathcal{Q}}\int_{Q^{**}}\phi_{Q}(\mathbf{x})\sup_{0<t\leq d(Q)^2}|K_t(f)(\mathbf{x})|\,dw(\mathbf{x})\leq C\|f\|_{L^1(dw)}+\|\sup_{t>0}|K_tf|\|_{L^1(dw)}.
\end{align*}
Hence, we have obtained
\begin{align*}
    \sum_{Q \in \mathcal{Q}}\|\phi_{Q}f\|_{H^{1}_{{\rm{loc}}, d(Q)}} \leq C\|f\|_{H^1_{L}},
\end{align*}
therefore, by Proposition~\ref{propo:Goldberg} we get
\begin{align*}
    \phi_{Q}(\mathbf{x})f(\mathbf{x})=\sum_{j=0}^{\infty}c_{j,Q}(\mathbf{x})a_{j,Q}(\mathbf{x})
\end{align*}
where $a_{j,Q}$ are atoms of local Hardy space $H^{1,{\rm{at}}}_{{\rm{loc}}, d(Q)}$ (see Definition~\ref{def:atomic_Dunkl} and Proposition~\ref{propo:Goldberg}) and
\begin{align*}
    \sum_{Q \in \mathcal{Q}}\sum_{j=0}^{\infty}|c_{j,Q}| \leq C\|f\|_{H^1_{L}}.
\end{align*}
Moreover, by Proposition~\ref{propo:Goldberg}, $\supp \phi_{Q} f \subseteq Q^{*}$ implies $\supp a_{j,Q} \subseteq Q^{****}$. Consequently, by Definition~\ref{def:atomic}, each $a_{j,Q}$ is an atom of $H^{1,{\rm at}}_{\mathcal{Q}}$.

\subsection{Proof of the inequality \texorpdfstring{$\|f\|_{H^1_{L}} \leq C\|f\|_{H^{1,{\rm{at}}}_{\mathcal{Q}}} $}{second}} It is enough to check if there is a constant $C>0$ such that for all atoms $a(\mathbf{x})$ of $H^{1,{\rm{at}}}_{\mathcal{Q}}$ we have $\|a\|_{H^1_{L}} \leq C$. Suppose that $a(\mathbf{x})$ is associated with a cube $Q \in \mathcal{Q}$. We write 
\begin{equation}\label{eq:finite_cubes}
a=\sum_{Q' \in \mathcal{Q}}\phi_{Q'}a . 
\end{equation} 
Thanks to~\eqref{eq:finite_overlap} and the fact that $\supp a \subseteq Q^{****}$, there is a number $M>0$ independent of $Q$ such that in~\eqref{eq:finite_cubes} there are at most $M$ nonzero summands with $d(Q') \sim d(Q)$. Let $\ell \geq 0$ be the smallest positive integer such that $d(Q') \geq 2^{-\ell/2}d(Q)$ for all such a cubes in~\eqref{eq:finite_cubes}. Clearly, thanks to~\eqref{eq:finite_overlap}, $\ell$ is independent of $a$ and $Q \in \mathcal{Q}$.  We write
\begin{align*}
    \|a\|_{H^1_{L}} \leq \|\sup_{0<t \leq 2^{-\ell}d(Q)^2}|K_ta|\|_{L^1(dw)}+\|\sup_{t> 2^{-\ell}d(Q)^2}|K_ta|\|_{L^1(dw)}=:I_1+I_2.
\end{align*}
Further,
\begin{align*}
    I_1 \leq \|\sup_{0<t \leq 2^{-\ell}d(Q)^2}|(K_t-H_t)a|\|_{L^1(dw)}+\|\sup_{0<t \leq 2^{-\ell}d(Q)^2}|H_ta|\|_{L^1(dw)}.
\end{align*}
Thanks to the fact that atom $a$ is, by definition, an atom for $H^{1}_{{\rm{loc}},d(Q)}$, we have
\begin{align*}
    \|\sup_{0<t \leq 2^{-\ell}d(Q)^2}|H_ta|\|_{L^1(dw)} \leq \|\sup_{0<t \leq d(Q)^2}|H_ta|\|_{L^1(dw)} \leq C.
\end{align*}
Thanks to~\eqref{eq:difference} and~\eqref{eq:finite_cubes}, we get
\begin{align*}
    \|\sup_{0<t \leq 2^{-\ell}d(Q)^2}|(K_t-H_t)\sum_{Q' \in \mathcal{Q}}(\phi_{Q'}a)|\|_{L^1(dw)} &\leq \sum_{Q' \in \mathcal{Q}}  \|\sup_{0<t \leq d(Q')^2}|(K_t-H_t)(\phi_{Q'}a)|\|_{L^1(dw)}\\&\leq C\sum_{Q' \in Q}\|\phi_{Q'}a\|_{L^1(dw)} \leq CM\|a\|_{L^1(dw)} \leq C.
\end{align*}
In order to estimate $I_2$, we repeat the argument presented in the proof of~\eqref{eq:outside_sum}. We provide details. We write
\begin{equation}\label{eq:j_split}
    I_2 \leq \sum_{j=-\ell}^{\infty}\|\sup_{2^{j}d(Q)^{2} < t \leq 2^{j+1}d(Q)^2}|K_ta|\|_{L^1(dw)}.
\end{equation}
By the semigroup property and Theorem~\ref{theorem:heat} together with~\eqref{eq:kernels_compare} for
\begin{align*}
2^{j}d(Q)^2 < t  \leq 2^{j+1}d(Q)^2    
\end{align*} 
we have
\begin{align*}
    &\int_{\mathbb{R}^N}k_t(\mathbf{x},\mathbf{y})|a(\mathbf{y})|\,dw(\mathbf{y})\\&=\int_{\mathbb{R}^N}\int_{\mathbb{R}^N}k_{t-2^{j-1}d(Q)^2}(\mathbf{x},\mathbf{z})k_{2^{j-1}d(Q)^2}(\mathbf{z},\mathbf{y})\,dw(\mathbf{z})|a(\mathbf{y})|\,dw(\mathbf{y})\\&\leq C \int_{\mathbb{R}^N}\int_{\mathbb{R}^N}\frac{1}{w(B(\mathbf{z},2^{j/2}d(Q)))}e^{-cd(\mathbf{x},\mathbf{z})^2/(2^{j+1}d(Q)^2)}k_{2^{j-1}d(Q)^2}(\mathbf{z},\mathbf{y})\,dw(\mathbf{z})|a(\mathbf{y})|\,dw(\mathbf{y}).
\end{align*}
Therefore, integrating over the $\mathbf{x}$-variable, we obtain
\begin{align*}
    &\int_{\mathbb{R}^N}\sup_{2^{j}d(Q)^2 \leq t <2^{j+1}d(Q)^2}\int_{\mathbb{R}^N}k_t(\mathbf{x},\mathbf{y})|a(\mathbf{y})|\,dw(\mathbf{y})\,dw(\mathbf{x})\\&\leq C \int_{\mathbb{R}^N}|a(\mathbf{y})|\int_{\mathbb{R}^N}k_{2^{j-1}d(Q)^2}(\mathbf{z},\mathbf{y})\,dw(\mathbf{z})\,dw(\mathbf{y}).
\end{align*}
Consequently, by condition~\eqref{eq:D} and~\eqref{eq:j_split}, we get
\begin{align*}
    I_2 \leq C\sum_{j=-\ell}^{\infty}j^{-1-\varepsilon}\|a\|_{L^1(dw)} \leq C'\|a\|_{L^1(dw)} \leq C'.
\end{align*}

\section{Verification of conditions\texorpdfstring{~\eqref{eq:finite_overlap},~\eqref{eq:D}, and~\eqref{eq:K}}{(F), (D), and (K)}} \label{sec:verification}

Let us note that the condition~\eqref{eq:finite_overlap} is already checked, see Fact~\ref{fact:F}.

\subsection{Verification of condition\texorpdfstring{~\eqref{eq:D}}{(D)}}
\begin{lemma}
There is a constant $C>0$ such that for all $\mathbf{y} \in \mathbb{R}^N$ and $t>0$ we have
\begin{equation}\label{eq:analytic}
    \langle Lk_{t}(\cdot,\mathbf{y}),k_t(\cdot,\mathbf{y})\rangle \leq \frac{C}{tw(B(\mathbf{y},\sqrt{t}))}.
\end{equation}
\end{lemma}

\begin{proof}
Thanks to the fact that operator $L$ is positive and self-adjoint, we have that the semigroup $\{K_{t}\}_{t \geq 0}$ is analytic on $L^2(dw)$, so the operator $LK_{t/2}$ is bounded on $L^2(dw)$ for all $t>0$. Therefore, by the semigroup property and the definition of $L$ (here $L_{\mathbf{x}}$ denotes the action of $L$ with respect to $\mathbf{x}$-variable) we have
\begin{equation}\label{eq:LKt/2}
    L_{\mathbf{x}}k_t(\mathbf{x},\mathbf{y})=L_{\mathbf{x}}\int_{\mathbb{R}^N}k_{t/2}(\mathbf{x},\mathbf{z})k_{t/2}(\mathbf{z},\mathbf{y})\,dw(\mathbf{z})=((LK_{t/2})k_{t/2}(\cdot,\mathbf{y}))(\mathbf{x}).
\end{equation}
Consequently, by the Cauchy--Schwarz inequality we have
\begin{equation}\label{eq:semigroup_scalar}
\begin{split}
    \langle Lk_{t}(\cdot,\mathbf{y}),k_t(\cdot,\mathbf{y})\rangle&=\langle LK_{t/2}(k_{t/2}(\cdot,\mathbf{y}))(\cdot),k_t(\cdot,\mathbf{y})\rangle \\&\leq \|k_{t}(\cdot,\mathbf{y})\|_{L^2(dw)}\|LK_{t/2}(k_{t/2}(\cdot,\mathbf{y}))(\cdot)\|_{L^2(dw)}.
\end{split}
\end{equation}
By Theorem~\ref{theorem:heat} and~\eqref{eq:kernels_compare} we obtain
\begin{equation}\label{eq:semi_1}
    \|k_{t}(\cdot,\mathbf{y})\|_{L^2(dw)} \leq \frac{C}{w(B(\mathbf{y},\sqrt{t}))^{1/2}}.
\end{equation}
Moreover,  holomorphy of $\{K_t\}_{t \geq 0}$ together with Theorem~\ref{theorem:heat} lead us to
\begin{equation}\label{eq:semi_2}
    \|LK_{t/2}(k_{t/2}(\cdot,\mathbf{y}))(\cdot)\|_{L^2(dw)} \leq C\frac{1}{t}\|k_{t/2}(\cdot,\mathbf{y})\|_{L^2(dw)} \leq C'\frac{1}{tw(B(\mathbf{y},\sqrt{t}))^{1/2}}.
\end{equation}
The claim is a consequence of~\eqref{eq:semigroup_scalar} together with~\eqref{eq:semi_1} and~\eqref{eq:semi_2}.
\end{proof}

Now we are ready prove that the condition~\eqref{eq:D} holds for the potential $V$ satisfying the reverse H\"older inequality~\eqref{eq:reverse_Holder}. Fix $\mathbf{y} \in \mathbb{R}^N$ and $0<t \leq d(Q)^2$. For any $r>0$ (it will be chosen later), by Cauchy--Schwarz inequality, ~\eqref{eq:kernels_compare}, and Theorem~\ref{theorem:heat} we obtain
\begin{align*}
    I=\left(\int_{\mathbb{R}^N}k_t(\mathbf{x},\mathbf{y})\,dw(\mathbf{x})\right)^{2} &\leq 2\left(\int_{\|\mathbf{x}-\mathbf{y}\| \leq r}k_t(\mathbf{x},\mathbf{y})\,dw(\mathbf{x})\right)^{2}+2\left(\int_{\|\mathbf{x}-\mathbf{y}\| > r}k_t(\mathbf{x},\mathbf{y})\,dw(\mathbf{x})\right)^{2} \\&\leq Cw(B(\mathbf{y},r))\int_{\|\mathbf{x}-\mathbf{y}\| \leq r}k_t(\mathbf{x},\mathbf{y})^2\,dw(\mathbf{x})+Ctr^{-2}.
\end{align*}
By~\eqref{eq:LKt/2} and the comment above~\eqref{eq:LKt/2} we have $k_{t}(\cdot,\mathbf{y}) \in \mathcal{D}(L)$. Therefore
\begin{align*}
    \mathbf{Q}(k_{t}(\cdot, \mathbf{y}),k_{t}(\cdot,\mathbf{y}))=\langle Lk_t(\cdot,\mathbf{y}),k_{t}(\cdot,\mathbf{y}) \rangle.
\end{align*}
Consequently, using~\eqref{eq:Shen_C}, then Theorem~\ref{theo:Fefferman-Phong}, we get
\begin{equation}\label{eq:k_square}
\begin{split}
    I &\leq Cw(B(\mathbf{y}, r))m(\mathbf{y})^{-2}(1+rm(\mathbf{y}))^{\frac{2\kappa}{1+\kappa}}\int_{\mathbb{R}^N}k_t(\mathbf{x},\mathbf{y})^2m(\mathbf{x})^2\,dw(\mathbf{x})+Ctr^{-2}\\&\leq Cw(B(\mathbf{y}, r))m(\mathbf{y})^{-2}(1+rm(\mathbf{y}))^{\frac{2\kappa}{1+\kappa}}\langle Lk_t(\cdot,\mathbf{y}),k_{t}(\cdot,\mathbf{y}) \rangle +Ctr^{-2}.
\end{split}
\end{equation} 
By~\eqref{eq:analytic} and~\eqref{eq:growth} we get
\begin{equation}\label{eq:without_plug}
\begin{split}
    I &\leq C\frac{w(B(\mathbf{y},r))}{tw(B(\mathbf{y},\sqrt{t}))}m(\mathbf{y})^{-2}(1+rm(\mathbf{y}))^{\frac{2\kappa}{1+\kappa}}+Ctr^{-2} \\&\leq C(r^{\mathbf{N}}t^{-\mathbf{N}/2}+r^{N}t^{-N/2})m(\mathbf{y})^{-2}(1+rm(\mathbf{y}))^{\frac{2\kappa}{1+\kappa}}+Ctr^{-2}.
\end{split}
\end{equation}
If we plug in $r=t^{\frac{1+\varepsilon}{2}}m(\mathbf{y})^{\varepsilon}$, we get
\begin{align*}
    I \leq C(t^{\mathbf{N}\varepsilon/2-1}m(\mathbf{y})^{\mathbf{N}\varepsilon-2}+t^{N\varepsilon/2-1}m(\mathbf{y})^{N\varepsilon-2})(1+t^{1/2+\varepsilon/2}m(\mathbf{y})^{1+\varepsilon})^{2\kappa/(1+\kappa)}+Ct^{-\varepsilon}m(\mathbf{y})^{-2\varepsilon},
\end{align*}
so if we take $\varepsilon$ small enough,
we get
\begin{align*}
    I \leq Ct^{-\varepsilon_1}m(\mathbf{y})^{-2\varepsilon_1} \text{ for some }\varepsilon_1>0,
\end{align*}
which, thanks to the fact that for $\mathbf{y} \in Q^{****}$ we have $m(\mathbf{y}) \sim d(Q)^{-1}$ (see Fact~\ref{fact:m_d_compare}), ends the proof.

\subsection{Verification of condition~\texorpdfstring{~\eqref{eq:K}}{(K)}}
Thanks to H\"older's inequality with the exponent $q$ from~\eqref{eq:reverse_Holder} we have
\begin{equation}\label{eq:holder_heat}
\begin{split}
    I&= \int_0^{2t}\int_{Q^{***}}V(\mathbf{y})\mathcal{G}_{2s/c}(\mathbf{x},\mathbf{y})\,dw(\mathbf{y})\,ds\\& \leq \int_0^{2t}\left(\frac{1}{w(Q^{***})}\int_{Q^{***}}V(\mathbf{y})^q\,dw(\mathbf{y})\right)^{1/q}w(Q^{***})^{1/q}\left(\int_{Q^{***}}\mathcal{G}_{2s/c}(\mathbf{x},\mathbf{y})^{q'}\,dw(\mathbf{y})\right)^{1/q'}\,ds.
\end{split}
\end{equation}
Furthermore, by the definition of $\mathcal{G}_{2s/c}$ (see~\eqref{eq:mathcal_G}) we have
\begin{equation}\label{eq:heat_Lq'}
    \int_{Q^{***}}\mathcal{G}_{2s/c}(\mathbf{x},\mathbf{y})^{q'}\,dw(\mathbf{y}) \leq C\int_{Q^{***}}\frac{1}{w(B(\mathbf{y},\sqrt{s}))^{q'-1}}\frac{1}{w(B(\mathbf{x},\sqrt{s}))}e^{-cq'd(\mathbf{x},\mathbf{y})^2/(2s)}\,dw(\mathbf{y}).
\end{equation}
Note that for $\mathbf{y} \in Q^{***}$ we have $w(B(\mathbf{y},d(Q))) \sim w(Q^{***})$, therefore,
\begin{equation}\label{eq:heat_Lq'_cube}
\begin{split}
    &w(Q^{***})^{1/q}\left(\int_{Q^{***}}\mathcal{G}_{2s/c}(\mathbf{x},\mathbf{y})^{q'}\,dw(\mathbf{y})\right)^{1/q'} \\&\leq C\left(\int_{Q^{***}}\frac{w(Q^{***})^{q'/q}}{w(B(\mathbf{y},\sqrt{s}))^{q'-1}}\frac{1}{w(B(\mathbf{x},\sqrt{s}))}e^{-cq'd(\mathbf{x},\mathbf{y})^2/(2s)}\,dw(\mathbf{y})\right)^{1/q'}\\&\leq C\left(\int_{Q^{***}}\frac{w(B(\mathbf{y},d(Q)))^{q'/q}}{w(B(\mathbf{y},\sqrt{s}))^{q'-1}}\frac{1}{w(B(\mathbf{x},\sqrt{s}))}e^{-cq'd(\mathbf{x},\mathbf{y})^2/(2s)}\,dw(\mathbf{y})\right)^{1/q'}.
\end{split}
\end{equation}
Thanks to~\eqref{eq:growth} we have (let us remind that $\sqrt{s} \leq \sqrt{2t} \leq \sqrt{2}d(Q)$ by assumption of~\eqref{eq:K})
\begin{align*}
    \frac{w(B(\mathbf{y},d(Q)))^{q'/q}}{w(B(\mathbf{y},\sqrt{s}))^{q'-1}}=\frac{w(B(\mathbf{y},d(Q)))^{q'/q}}{w(B(\mathbf{y},\sqrt{s}))^{q'/q}} \leq C\left(\frac{d(Q)}{\sqrt{s}}\right)^{(\mathbf{N}q')/q}.
\end{align*}
Consequently,~\eqref{eq:holder_heat} and~\eqref{eq:heat_Lq'_cube} lead us to
\begin{align*}
    I \leq C\int_0^{2t}\left(\frac{d(Q)}{\sqrt{s}}\right)^{\mathbf{N}/q}\left(\frac{1}{w(Q^{***})}\int_{Q^{***}}V(\mathbf{y})^q\,dw(\mathbf{y})\right)^{1/q}\,ds,
\end{align*}
so, thanks to the reverse H\"older inequality~\eqref{eq:reverse_Holder} and the fact that $q>\max(1,\frac{\mathbf{N}}{2})$ we have
\begin{align*}
    I \leq Ct\left(\frac{d(Q)}{\sqrt{t}}\right)^{\mathbf{N}/q}\frac{1}{w(Q^{***})}\int_{Q^{***}}V(\mathbf{y})\,dw(\mathbf{y}) \leq C\left(\frac{d(Q)}{\sqrt{t}}\right)^{\mathbf{N}/q-2},
\end{align*}
where in the last step the fact that the measures $\mu$ and $w$ are doubling and the definition of $Q \in \mathcal{Q}$ by the stopping-time condition~\eqref{eq:stopping_time}, that means
\begin{align*}
    \frac{1}{w(Q^{***})}\int_{Q^{***}}V(\mathbf{y})\,dw(\mathbf{y}) \leq C\frac{1}{w(Q)}\int_{Q}V(\mathbf{y})\,dw(\mathbf{y}) \leq Cd(Q)^{-2}.
\end{align*}
The proof is finished (we set $\delta=1-\frac{\mathbf{N}}{2q}$).

{\bf Acknowledgment.} The author would like to thank Jacek Dziuba\'nski for careful reading of the text and his helpful comments and suggestions.

\end{document}